\RequirePackage{fix-cm}
\documentclass[smallextended,numbook,runningheads]{svjour3}
\smartqed  
\usepackage{verbatim}
\usepackage{xcolor}
\usepackage{graphicx}
\usepackage{amsmath,amssymb,bm}
\usepackage{graphicx,graphics,color}
\definecolor{refgreen}{rgb}{0,0.5,0}
\usepackage[colorlinks,linkcolor=blue,citecolor=blue]{hyperref}
\usepackage{tikz}
\usepackage{booktabs,subfig} 

\usepackage[utf8]{inputenc}
\journalname{...}
\date{ \phantom{b} \vspace{45mm}\phantom{e}}
\def\makeheadbox{\hspace{-4mm}Version of \today}

\usepackage{color}
\newcommand{\red}[1]{\textcolor[rgb]{1.00,0.00,0.00}{#1}}

\def\half{\frac{1}{2}}

\newcommand{\bR}{{\mathbb R}}
\newcommand{\bS}{{\mathbb S}}
\newcommand{\diag}{\operatorname{diag}}

\newcommand\bfd{{\mathbf d}}
\newcommand\bfe{{\mathbf e}}

\newcommand\bff{{\mathbf f}}

\newcommand\bfn{{\mathbf n}}

\newcommand\bfu{{\mathbf u}}
\newcommand\bfv{{\mathbf v}}
\newcommand\bfw{{\mathbf w}}
\newcommand\bfx{{\mathbf x}}

\newcommand\bfy{{\mathbf y}}
\newcommand\bfz{{\mathbf z}}
\newcommand\bfA{{\mathbf A}}

\newcommand\bfH{{\mathbf H}}

\newcommand\bfK{{\mathbf K}}
\newcommand\bfM{{\mathbf M}}

\newcommand\bfV{{\mathbf V}}

\newcommand\calE{{\mathcal E}}

\newcommand\calO{{\mathcal O}}

\newcommand\andquad{\quad\hbox{ and }\quad}

\renewcommand{\d}{\textnormal{d}}

\newcommand{\D}{D}
\newcommand{\ga}{\gamma}
\newcommand{\Ga}{\varGamma}

\newcommand{\laplace}{\Delta}

\newcommand{\nbg}{\nabla_{\varGamma}}
\newcommand{\nbgh}{\nabla_{\varGamma_h}}
\newcommand{\mat}{\partial^{\bullet}}

\newcommand{\diff}{\frac{\d}{\d t}}
\newcommand{\eps}{\varepsilon}

\newcommand{\nb}{\nabla}

\newcommand{\pa}{\partial}
\newcommand{\R}{\mathbb{R}}

\newcommand{\spn}{\textnormal{span}}
\def \t {(t)}

\def \to {\rightarrow}
\newcommand{\tr}{\textnormal{tr}}
\newcommand{\vphi}{\varphi}

\newcommand{\Ih}{\widetilde{I}_h}

\newcommand{\us}{\bfu^\ast}
\newcommand{\vs}{\bfv^\ast}
\newcommand{\xs}{\bfx^\ast}

\newcommand{\dotus}{\dot\bfu^\ast}

\newcommand{\dotxs}{\dot\bfx^\ast}

\newcommand{\xls}{\bfx_\ast}

\newcommand{\eu}{\bfe_\bfu}
\newcommand{\ev}{\bfe_\bfv}
\newcommand{\ex}{\bfe_\bfx}

\newcommand{\doteu}{\dot\bfe_\bfu}

\newcommand{\dotex}{\dot\bfe_\bfx}

\newcommand{\du}{\bfd_\bfu}
\newcommand{\dv}{\bfd_\bfv}

\newcommand{\dof}{N}

\newcommand{\doi}[1]{doi:\href{https://doi.org/#1}{\texttt{#1}}}

\newcommand{\rewriteoff}{\color{black}}

\newcommand{\bffb}{{\mathbf f}}

\begin{document}

\title{Error estimates for a surface finite element method for anisotropic mean curvature flow}

\titlerunning{Error estimates for anisotropic mean curvature flows}        

\author{Klaus Deckelnick \and Harald Garcke \and Bal\'azs Kov\'acs}

\authorrunning{K.~Deckelnick, H.~Garcke, and B.~Kov\'acs} 

\institute{
	Klaus Deckelnick \at
	Institut f\"ur Analysis und Numerik, Otto-von-Guericke-Universit\"at Magdeburg, 
	Universitätsplatz 2, 39106 Magdeburg, Germany; 
	\email{klaus.deckelnick@ovgu.de} \\
	H.~Garcke\at
	Faculty of Mathematics, University of Regensburg, 
	Universit{\"a}tsstr. 31, 93040 Regensburg, Germany; 
	\email{harald.garcke@ur.de}\\
	B.~Kov\'{a}cs \at
	Institute of Mathematics, Paderborn University,
	Warburgerstr.~100., 31098 Paderborn, Germany; 
	\email{balazs.kovacs@math.uni-paderborn.de}\\
}

\date{\today}

\maketitle

\begin{abstract}
	Error estimates are proved for an evolving surface finite element semi-discretization for anisotropic mean curvature flow of closed surfaces.
	For  the geometric surface flow, a system coupling the anisotropic evolution law to parabolic evolution equations for the surface normal and normal velocity is derived, which then serve as the basis for the proposed numerical method.
	The algorithm for anisotropic mean curvature flow
	is proved to be convergent in the $H^1$-norm with optimal-order for finite elements of degree at least two. 
	Numerical experiments are presented to illustrate and complement our theoretical results.
\end{abstract}


\section{Introduction}

In this paper we propose an evolving surface finite element semi-discretization of anisotropic mean curvature flow
for closed surfaces $\Ga$ (of dimension at most three). We prove optimal-order $H^1$-norm error estimates  for finite elements of degree at least two, and over time intervals on which the surface evolving under the  anisotropic mean curvature flow remains sufficiently regular.

The proposed algorithm is based on a system coupling the anisotropic surface flow to parabolic evolution equations for the surface normal $\nu$ and normal velocity $V$. These equations are for the first time  derived here. 
This approach was pioneered in 1984 by Huisken \cite{Huisken1984} for mean curvature flow, and stays to be an important tool for the analysis of geometric flows ever since. The same idea has also been proved to be useful in the numerical analysis of geometric flows, see, e.g.~\cite{MCF,KLL_Willmore,MCF_soldriven,MCF_generalised,MCF_codim,pq_method,MCFdiff,bulksurface_coupling}.

	The anisotropic mean curvature flow studied in this paper can be considered as the gradient flow of the anisotropic surface energy
\begin{equation}\label{eq:anenergy}
	\calE_{\ga}(\Ga) = \int_\Gamma \gamma(\nu) ,
\end{equation}
where $\Gamma$ is a closed orientable $C^1$-hypersurface 
in ${\mathbb R}^{d+1}$, $d\geq1$, with a continuous unit normal field $\nu$,
and $\gamma \colon \bS^{d} \to {\mathbb R}_{>0}$ is a given anisotropic energy density. 
It is helpful, to extend $\gamma$ to a function on $\bR^{d+1}$
as a positively one-homogeneous function.

In order to visualize the surface energy, it is convenient to consider a generalized  isoperimetric problem  for the surface energy $\calE_{\ga}$.  One wants to find the
shape, which minimizes $\calE_{\ga}$ under all shapes with a given 
enclosed volume. In order to do so, one defines the dual function
\begin{equation*}
	\gamma^\ast ( q) = \sup_{ p\in\bR^{d+1}\setminus\{ 0\}}
	\frac{ p\,\cdot\,q}{\gamma( p)} \qquad \forall  q \in \bR^{d+1}\,.
\end{equation*}
Then the solution of the isoperimetric problem is, up to a scaling,
the Wulff shape, see \cite{Gurtin1993} for details: 
\begin{equation*}
	\mathcal{W} = \{ q\in\bR^{d+1} : \gamma^\ast( q) \le 1\}\,.
\end{equation*}
This is the $1$-ball of $\gamma^\ast$, and we also define the $1$-ball
of $\gamma$
\begin{equation*}
	\mathcal{F} = \{ p \in\bR^d : \gamma( p) \le 1\}\,,
\end{equation*}
which is called Frank diagram. We refer to Figure~\ref{fig:BGNwulff}
for examples. 

In many applications, flows are of interest that decrease the total anisotropic energy $	\calE_{\ga}$.
To derive such flows, we consider    a family of closed, smooth and oriented hypersurfaces $(\Gamma(t))_{t \in [0,T]}$
and compute, see \cite{Giga_2006,BDGP2023} for a proof, 
\begin{equation*}
	\frac{\d}{\d t} \calE_{\ga} (\Gamma(t)) = \int_{\Gamma(t)} H_\gamma V.
\end{equation*}
Here, $V$ is the normal velocity of $\Gamma(t)$ and $	H_\ga = \nbg \cdot \big( \ga'(\nu) \big)$ is the anisotropic mean curvature. For $\nu$ we use the outward unit normal which leads to
the sign convention that $H_\ga$ is positive for convex surfaces.
We hence obtain that flows of the form
\begin{equation} \beta(\nu) V = - H_\ga \qquad \mbox{ on } \Gamma(t)
	\label{eq:anisolaw}
	\end{equation}
decrease the energy whenever the kinetic coefficient $\beta$ is positive. It can be shown that \eqref{eq:anisolaw} is the gradient flow of the anisotropic surface energy 
in the case that one uses a weighted $L^2$-inner product, where the weight is given by $\beta$, see, e.g., \cite{GarckeNSW2008,LauxSU2025}.
The above geometric evolution law has many applications both in mathematics as well as   in physics and engineering.
In materials science, surface motion is classified as geometric if the normal velocity of the surface depends only on the position and the local shape of the surface. In  particular,
no long range effects play a role. In a review article, Taylor, Cahn and Handwerker \cite{TAYLOR19921443} discuss that in materials science applications  many interface motion problems can be modeled as geometric and name crystal growth, certain types of phase change problems, grain growth and chemical etching as examples. Due to the fact that crystals are anisotropic these 
geometric evolution problems typically lead to laws which are similar as \eqref{eq:anisolaw}. We refer to \cite{TAYLOR19921443,Gurtin1993,BellettiniPaolini_1996,Giga_2006} for more details on anisotropic energies and anisotropic curvature flows in materials science and geometry.

Beside the case $\beta=1$, the kinetic term   $\beta=\frac 1\gamma$ is of particular interest.
As discussed in \cite{BellettiniPaolini_1996}  one can consider the evolution in relative geometry, i.e.,   all quantities are referred to the given so-called Finsler metric given by $\gamma$  representing the anisotropy. In this case,  one obtains a gradient flow which is \eqref{eq:anisolaw} with $\beta=\frac 1\gamma$.
Also for this particular choice of the kinetic coefficient one obtains that  the Wulff shape is shrinking in a self-similar way, see Theorem 1.7.3 in \cite{Giga_2006}. 
 Let us mention that for curves it is known that always self-similar solutions
exist, see \cite{DohmenGN1996}. However, in the case $\beta\neq \frac 1\gamma$   no explicit solution is known.

Analytically, anisotropic mean curvature flow is well studied. An important first result is due to Chen, Giga and Goto \cite{ChenGG1991}, who showed existence and uniqueness of viscosity solutions. Almgren, Taylor and Wang \cite{ATW1993} used an implicit time discretization  in the form of a minimizing movement scheme to show existence of solutions to 
\eqref{eq:anisolaw}. They also showed a short time existence result in the case that the anisotropy $\gamma$ is smooth.
A refined analysis of classical solutions in the curve case, including results on the long time behaviour are, due to Gage \cite{Gage1993}, see also \cite{MikulaS2001}.
By considering the evolution  $ V = - \gamma(\nu) H_\ga$ in relative geometry, see the discussion above, 
Bellettini and Paolini \cite{BellettiniPaolini_1996} 
 studied the  evolution
law using different approaches, such as the variational method of Almgren--Taylor--Wang,
the Hamilton--Jacobi approach, and an approximation using an anisotropic Allen--Cahn equation.
Chambolle and Novaga \cite{ChambolleN2007} 
gave simple proofs of convergence of the  Almgren--Taylor--Wang variational
approach, and Merriman--Bence--Osher algorithm relying on the theory of viscosity solutions. 
We refer to the recent paper \cite{LauxSU2025} who used a weak formulation of the anisotropic mean curvature operator in the BV-setting, which was introduced 
by \cite{CicaleseNP2010,GarckeS2011}, to show a weak-strong uniqueness result.



%
%
%
%

The aim of this paper is to derive and analyze a finite element method
for a parametric approach to solve \eqref{eq:anisolaw}. Using an idea originally
developed in \cite{MCF} for the (isotropic) mean curvature flow,
the scheme discretizes a system of evolution equations coupling the
position vector $X$, the surface normal $\nu$ and the normal velocity
$V$. These equations are anisotropic versions of identities derived by
Huisken in \cite{Huisken1984}. The resulting system is discretized in space with
the help of the evolving surface finite element method \cite{DziukElliott_ESFEM} using
continuous, piecewise polynomials of degree $k \geq 2$. 

\vskip 3mm
{\bf Main result:} Under an
appropriate smoothness assumption on the exact solution  and for
sufficiently small spatial grid size $h$
it holds:
\begin{align*}
	\max_{t \in [0,T]}    \|X_h^\ell(\cdot,t) - 
	X(\cdot,t)\|_{H^1(\Ga^0)^{d+1}} \leq &\ Ch^k , \\
	\max_{t \in [0,T]}  \|(\nu_h^L,V_h^L)(\cdot,t) - 
	(\nu,V)(\cdot,t)\|_{H^1(\Ga[X(\cdot,t)])^{d+2}} \leq &\ C h^k.
\end{align*}
In the above, $\Ga[X(\cdot,t)]$ is the exact surface at time
$t$, $\Ga^0=\Ga[X(\cdot,0)]$, and $X_h^\ell,\nu_h^L,V_h^L$ are 
suitable lifts of the discrete solution $X_h,\nu_h,V_h$ to the exact surfaces $\Ga^0$ and
$\Ga[X(\cdot,t)]$, respectively, see Section~\ref{section:ESFEM} for precise definitions and Theorem~\ref{theorem: error estimates - modified aMCF} for an exact formulation of the main result.
\vskip 2mm

For \emph{anisotropic} geometric flows numerical methods based on finite elements have started with Dziuk \cite{Dziuk_aCSF_1999}, showing semi-discrete error estimates for anisotropic curve shortening flow, and \cite{DeckenickDziuk_fullydiscr_aCSF_2002} showing fully discrete error estimates in the graph setting.
For  the crystalline flow of polygonal curves   different  methods have been  developed in \cite{RoosenT1992,GigaGiga_2000}.  A numerical method based on a formulation which redistributes points tangentially and which  uses a PDE for geometric quantities like curvature has been developed for curves in \cite{MikulaSevcovic_2001,MikulaSevcovic_2004}, while a regularised fully discrete algorithm was proposed in \cite{HausserVoigt_2006}. Pozzi \cite{Pozzi_aCSF_2007} has shown spatial error estimates for curves in higher codimension, also see \cite{Pozzi_anisotropic_area_2012} for some analytical results.
Recently, two new schemes \cite{DeckelnickNurnberg_aCSF_2023,DeckelnickNurnberg_new_aCSF_2023} were proposed for anisotropic curve shortening flow.
A numerical method for the anisotropic mean curvature flow of a graph is proposed in \cite{HoangBenes_aMCF_graph_2014}.
Barrett, Garcke, and N\"{u}rnberg have developed a method with tangential redistributions for anisotropic flows in \cite{BGN_a_plane_2008} for curves, in \cite{BGN_anisotropic_curves_2010} for anisotropic gradient flows for space curves, and in  \cite{BGN2012} for anisotropic elastic curves.

The first parametric finite element-based algorithms for anisotropic mean curvature flow of \emph{surfaces} were proposed in \cite{Pozzi_2008_aMCF} and \cite{BGN_anisotropic}. 
Algorithms for an\-iso\-tro\-pic \emph{Willmore} flow were proposed in \cite{Pozzi_2015_anisotropicWillmore,BGN_anisotropic,PerlPozziRumpf_2014}, and for anisotropic \emph{surface diffusion} in \cite{Burger_2005,HausserVoigt_2007,BGN_anisotropic,BaoLi_aSD_2023,BaoLi_aSD_2024} (the last method being structure-preserving).
For numerical methods for anisotropic high-order geometric flows see, e.g., \cite{HausserVoigt_2005}.
We refer to \cite{DeckelnickDE2005} for a survey up to 2005, and a more recent one \cite{BGN_survey} up to 2020.

For numerical methods based on level-sets or diffuse interfaces we refer, e.g., to \cite{GarckeNS1999,Li_etal_2009,ObermanOsherTakeiTsai,GraserKornhuberSack_2013,BGNanisoPF,Saalvalagli_etal_2021}.

For \emph{surfaces} evolving under an \emph{anisotropic} geometric flow, up to our knowledge, there are no error estimates for parametric approaches in the literature.


The paper is structured as follows: 

In the following section, we first present details on  surface calculus,  basic facts on anisotropy functions and on evolving hypersurfaces.
The governing evolution equations are then derived in in Section~\ref{section:anisotropic surface flows} together with 
weak formulations which will be the basis for the finite element method.
The evolving surface finite element method is described in Section~\ref{section:ESFEM}, the semi-discrete problem and the precise statement of the main result 
are stated in Section \ref{section:discretization and main result}.
In Section \ref{sec:erroranalysis} we prove the main result heavily relying on a matrix--vector formulation, consistency bounds and energy-type arguments.
In Section \ref{section:regularization} we introduce a regularized problem which will be helpful,  in particular, in situations where the anisotropy is nearly crystalline.
Section \ref{sec:numericalexperiments} presents numerical computations, beside others, rate of convergence simulations and simulations of anisotropic mean curvature flow with ``crystalline-like'' anisotropies  are given. Finally, an appendix give proofs for important formulas from geometric analysis.

\section{Preliminaries}
\subsection{\it Basic surface calculus} \label{sec:bsc}
Let $\Gamma \subset \mathbb R^{d+1}$ ($d=1,2,3$) be a smooth, embedded and closed hypersurface, oriented by a continuous unit normal field $\nu \colon \Gamma \rightarrow \mathbb R^{d+1}$. For
a function $u \colon \Gamma \rightarrow \mathbb R$ we denote by $\nabla_{\Ga} u \colon \Ga\to\R^{d+1}$ the \emph{tangential gradient} or \emph{surface gradient} of $u$ 
and write
\begin{equation*}
	\nbg u = \big(  \D_1 u, \D_2 u, \dotsc, \D_{d+1} u \big)^T ,
\end{equation*} 
so that we think of $\nbg u$ as a column vector. 
In the case of a vector-valued function $w=(w_1,\dotsc,w_{d+1})^T\colon \Ga\to\R^{d+1}$, we let
$\nabla_{\Ga}w=
(\nabla_{\Ga}w_1,
\dotsc,
\nabla_{\Ga}w_{d+1})$. 
Furthermore, we denote by $\nabla_{\Ga} \cdot w = \text{tr}(\nbg w)$ the \emph{surface divergence} of $w$, while we use the convention that the divergence of a matrix is computed column-wise. 
For a scalar function $u\colon \Gamma \rightarrow \mathbb R$ we let 
 $\varDelta_{\Ga} u = \nabla_{\Ga} \cdot \nabla_{\Ga} u$ the \emph{Laplace--Beltrami operator}. \\
The (extended) Weingarten map is defined by $A = \nbg \nu \in \R^{(d+1) \times (d+1)}$,
which is a symmetric matrix, whose eigenvalues are the principal curvatures $\kappa_1,\dotsc, \kappa_d$, and $0$ (corresponding to the eigenvector $\nu$).
The mean curvature of $\Ga$ is then given by $H = \sum_{j=1}^{d} \kappa_j = \tr(A)$, 
where we note that $H$ is positive for a sphere if the unit normal vector points outwards.  
See the review \cite{DeckelnickDE2005}, or \cite[Appendix~A]{Ecker2012}, etc.~for these notions.

\subsection{\it Anisotropy functions} \label{sec:anisotropy}

Let $\gamma \in C^0(\mathbb R^{d+1}, \mathbb R_{\geq 0}) \cap C^4( \R^{d+1} \setminus \{0\}, \R_{>0})$ be positively homogeneous of degree one, i.e.,
\begin{equation} \label{eq:homog}
		\ga(\lambda w) = \lambda \ga(w) \qquad \text{for all } w  \in \R^{d+1}, \lambda > 0 .
\end{equation}
Using \eqref{eq:homog} one verifies that
\begin{eqnarray}
		\ga'(\lambda w)& =& \ga'(w) , \qquad \ga''(\lambda w) = \frac{1}{\lambda}  \ga''(w) ,   \label{eq:aniso1}  \\
	\ga'(\lambda w) \cdot w & = & \ga(w) , \qquad 	\ga''(w) \, w = 0    \label{eq:aniso2}
\end{eqnarray}
for any $w \in \R^{d+1} \setminus \{0\}$, $\lambda > 0$, where $\cdot$  denotes the scalar product of two vectors.  Furthermore, it can be shown that there exists $c_1 \geq 0$ such that
\begin{eqnarray}  
| \gamma''(v_1) - \gamma''(v_2) |& \leq &  c_1 | v_1 - v_2 | \quad \forall \,  v_1,v_2 \in \mathbb R^{d+1}, \, \frac{1}{2} \leq  |v_1|,|v_2| \leq 2 , \label{eq:gamma2} \\
| \gamma'''(v_1) - \gamma'''(v_2) |& \leq &  c_1 | v_1 - v_2 | \quad \forall \,  v_1,v_2 \in \mathbb R^{d+1}, \, \frac{1}{2} \leq  |v_1|,|v_2| \leq 2 . \label{eq:gamma3} 
\end{eqnarray}
In what follows we shall assume that $\gamma$
is strongly convex in the sense that there exists $c_0 > 0$ such that
\begin{equation} \label{eq:convex}
		w \cdot \ga''(z) w \geq c_0 |w|^2  \qquad \text{for all } z,w  \in \R^{d+1}, \, |z|=1, \, z \cdot w = 0 .
\end{equation}
As an  immediate consequence of the above property we deduce that there exists $\delta>0$ such that for all
$\tilde z, z, w \in \R^{d+1}$ satisfying $ \frac{1}{2} \leq  | \tilde z |,  | z | \leq 2,  | \tilde z - z| < \delta$ and  $z \cdot w =0$:
\begin{equation} \label{eq:convex1}
w \cdot \ga''(\tilde z) w \geq \frac{c_0}{4} | w|^2. 
\end{equation} 
Indeed, we infer with the help of  \eqref{eq:aniso1}, \eqref{eq:convex} and \eqref{eq:gamma2}
\begin{eqnarray*}
w \cdot \ga''(\tilde z) w & = & w \cdot \ga''(z) w + w \cdot (\ga''(\tilde z) - \ga''(z)) w  \\
& = & \frac{1}{|z|} w \cdot \gamma''\bigl( \frac{z}{|z|} \bigr)w +  w \cdot (\ga''(\tilde z) - \ga''(z)) w \\
 & \geq &   \frac{c_0}{2}  | w |^2 - c_1 | \tilde z - z| \, | w |^2 
 \geq (\frac{c_0}{2}  - c_1 \delta) | w| ^2 = \frac{c_0}{4} | w|^2,
\end{eqnarray*}
provided that $\delta = \frac{c_0}{4 c_1}$.

\begin{example}
\label{example:anisotropies}
	(a) Any norm on $\R^{d+1}$ which is smooth outside $0$ is a possible anisotropy function $\gamma$ suitable for
	this paper. In particular, one can take an $\ell^p$-norm if $p\in(1,\infty)$.
	
	(b) In many papers, see, e.g., \cite{BGN_anisotropic_curves_2010,BGN_anisotropic,BGN_survey}, the anisotropy function $\ga(w) = \sqrt{w \cdot G w}$ with a positive definite and symmetric  matrix $G \in \R^{(d+1) \times (d+1)}$ is considered. Its derivatives are given by
	\begin{equation*}
		\ga'(w) = \frac{1}{\ga(w)}  G  w , \andquad
		\ga''(w) = \frac{1}{\ga(w)}  G   - \frac{1}{\ga(w)^3} \big( G  w\big) \otimes \big(G  w\big), 
	\end{equation*}
	in particular we then have $\ga''(w) w = 0$ for any $w \in \R^{d+1}$. 
	
(c)	The above anisotropy only leads to ellipsoidal Wulff shapes. Therefore, often combinations  of such anisotropy functions are considered.
Following \cite{BGN_anisotropic_curves_2010,BGN_anisotropic}, we now consider a larger class of surface energy densities, which
are given as suitable norms of the ellipsoidal anisotropies. In particular,
we choose
\begin{equation} \label{eq:hggenBGNaniso}
	\gamma( w) = \left(\sum_{\ell=1}^L\,
	\left[\gamma_{\ell}( w)\right]^r\right)^{\frac1r}, \qquad
	\gamma_\ell( w)= \sqrt{ w \cdot G_{\ell}\,w}\,,
	\qquad\ell=1,\ldots, L\,,
\end{equation}
so that
\[
\gamma'( w) = \left[\gamma( w)\right]^{1-r}\sum_{\ell=1}^L\,
\left[\gamma_{\ell}( w)\right]^{r-1} \gamma'_\ell( w)\,. 
\]
Here $r\in [1,\infty)$ and $ G_{\ell} \in \bR^{(d+1)\times (d+1)}$, 
$\ell=1,\ldots, L$, are symmetric and positive definite. It turns
out that most energies of relevance can be approximated by the above
class of energies. In particular, hexagonal and cubic anisotropies can be
modelled with appropriate choices of $r$, $L$ and 
$\{ G_{\ell}\}_{\ell=1}^L$, see Figure~\ref{fig:BGNwulff}.
\begin{figure}
	\centering
	\includegraphics[width=\textwidth, trim={0 220 0 220}, clip]{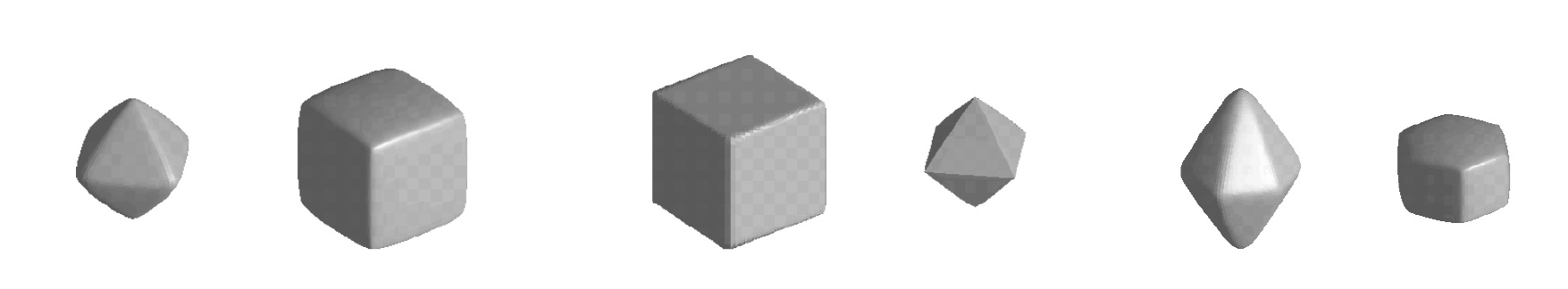} 
\caption{Frank diagram and Wulff shape in $\bR^3$ for a
	regularized $\ell^1$-norm, 
	$\gamma( w) = \sum_{\ell=1}^3\,[ \eps^2\,| w|^2+
	w_\ell^2(1-\eps^2)]^\frac12$, $\eps = 0.1$, left,
	a cubic anisotropy, 
	$\gamma( w) = [\sum_{\ell=1}^3\,[ \eps^2\,| w|^2+
	w_\ell^2(1-\eps^2)]^\frac r2]^\frac1r$, $\eps = 0.01$, $r =
	30$, middle, and a hexagonal anisotropy, 
	$\gamma( w) = \sum_{\ell=1}^4\,[  w\,\cdot\,
	\protect R_\ell^{T}\,
	\diag(1, \eps^2, \eps^2)\,\protect R_\ell^{}\, w ]^\frac12$, 
	$\eps = 0.1$, right. Here $\protect R_\ell^{}$, 
	$\ell = 1,\ldots,4$ are suitable rotation matrices, see \cite{BGN_anisotropic} for details.
}
\label{fig:BGNwulff}
\end{figure}%
The choice $r=1$ has the advantage
that it leads to more linear schemes. In particular, for $r=1$ we obtain
\begin{equation}
	\gamma''( w) = \sum_{\ell=1}^L\,
	\left[ 
	\frac{1}{\ga_\ell(w)}  G_\ell  - \frac{1}{\ga_\ell(w)^3} \big( G_\ell  w\big) \otimes \big(G_\ell  w\big)  \right] .
\end{equation}
\end{example}

For an oriented hypersurface $\Gamma$ with unit normal $\nu$ as in Section \ref{sec:bsc} the vector $\gamma'(\nu)$ is called \emph{Cahn--Hoffman vector}, while we denote by
\begin{equation}
	\label{eq:Cahn--Hoffman vector}
   	H_\ga = \nbg \cdot \big( \ga'(\nu) \big) 
\end{equation}
the anisotropic mean curvature $H_\ga$ of $\Gamma$. 

Note that in the isotropic case $\gamma(w)=|w|$ we have $\gamma'(\nu)=\nu$ so that
\begin{displaymath}
H_\gamma = \nabla_{\Gamma} \cdot \nu= \tr \nabla_{\Gamma} \nu = \tr A = H.
\end{displaymath}

The following identity will be important in deriving the governing evolution equation for the surface normal $\nu$.

\begin{lemma} Let $\Gamma \subset \mathbb R^{d+1}$ be a smooth hypersurface with unit normal $\nu$. Then
\begin{equation} \label{eq:three identities - H and nu-aniso}
			\nbg H_\ga = 
			\nbg \cdot \Big( \ga''(\nu) \nbg \nu \Big) + |A|_{\ga''}^2 \, \nu,
\end{equation}
where we abbreviated
\begin{equation} 	\label{eq:define MCF-like nonlinearity}
	|A|_{\ga''}^2 := |\nbg \nu|_{\ga''}^2 := \nbg \nu : ( \ga''(\nu) \nbg \nu ) 
\end{equation}
with  $A:B := \tr(A B^T)$ for $A,B \in \mathbb R^{(d+1) \times (d+1)}$. 
\end{lemma}
\begin{proof}
Using the definition $H_\ga = \nbg \cdot \big( \ga'(\nu) \big)$ and applying the gradient--divergence interchange formula Lemma~\ref{lemma:interchange formulas} \eqref{eqA:interchange formula - grad-div},
 we obtain 
\begin{align*}
	  \nbg H_\ga
	= &\ \nbg \big( \nbg \cdot ( \ga'(\nu) ) \big) \\
	= &\ \nbg \cdot \big( \nbg ( \ga'(\nu) ) \big)^T + (A : \nbg (\ga'(\nu)) ) \nu - A (\nbg (\ga'(\nu))) \nu  \\
	= &\ \nbg \cdot \big( \ga''(\nu) \nbg \nu \big) + \big(\nbg \nu : ( \ga''(\nu) \nbg \nu ) \big) \nu ,
\end{align*}
where the last equality follows  from
\begin{equation}
	\label{eq:symmetry of grad ga'}
	\nbg (\ga'(\nu)) = \nbg \nu \ga''(\nu)  = (\ga''(\nu) \nbg \nu)^T ,
\end{equation}
noting that both $A = \nbg \nu$ and $\ga''(\nu)$ are symmetric, and by the fact that $(\nbg (\ga'(\nu))) \nu = \nbg \nu \ga''(\nu) \nu = 0$ in view of \eqref{eq:aniso2}.
Recalling  \eqref{eq:define MCF-like nonlinearity} we obtain the desired form. \qed
\end{proof}

\begin{remark}
\label{remark:isotropic gamma}
In the isotropic case $\gamma(w)=|w|$, we have $\ga'(\nu) = \nu$, and $\ga''(\nu)  = I_{d+1}- \nu \nu^T$ so that
$|A|_{\ga''}^2  = \nbg \nu \cdot ( \ga''(\nu) \nbg \nu ) = \nbg \nu \cdot \nbg \nu = |A|^2$. Hence,  \eqref{eq:three identities - H and nu-aniso} reduces to the well-known identity
\begin{equation} \label{eq:three identities - H and nu}
\nbg H = \ \laplace_\Ga \nu + |A|^2 \nu .
\end{equation}
\end{remark}

\subsection{\it Evolving hypersurfaces} \label{sec:evolving}

Let $(\Gamma(t))_{t \in [0,T]}$ be a family of closed, smooth and oriented hypersurfaces with $G_T:= \bigcup_{t \in [0,T]} ( \Gamma(t) \times \lbrace t \rbrace)$. We assume that there exists a smooth mapping $X \colon \Gamma^0 \times [0,T] \rightarrow \mathbb R^{d+1}$ such that
$\Gamma(t)= \lbrace X(p,t) \, : \, p \in \Gamma^0 \rbrace$ 
and that $X(\cdot,t)$ is a diffeomorphism from $\Gamma^0$ to $\Gamma(t)$. In what follows we shall abbreviate 
$\Ga[X] = \Ga[X(\cdot,t)] = \Gamma(t)$ and denote by $v \colon G_T \rightarrow \mathbb R^{d+1}$ the velocity of material points given by
\begin{equation}
	\label{eq:velocity ODE}
	\partial_t X(p,t) = v(X(p,t),t).
\end{equation}
For a function $w \colon G_T \rightarrow \mathbb R$ we define the \emph{material derivative} (with respect to the parametrization $X$) by
$$
\mat w(x,t) = \diff \, w(X(p,t),t) \qquad \hbox{ for } \ x=X(p,t) \in \Gamma(t).
$$
The following  interchange formulas hold for sufficiently regular functions $u \colon G_T \to \R$ and vector fields $w\colon G_T \to \R^{d+1}$: 
\begin{subequations}
	\begin{align}
		\label{eq:interchange formula - mat-grad}
		\mat \big( \nbg u \big) = &\ \nbg ( \mat u ) - \big( \nbg v  - \nu \nu^T (\nbg v)^T \big) \nbg u , \\
		\label{eq:interchange formula - mat-div}
		\mat \big( \nbg \cdot w \big) = &\ \nbg \cdot ( \mat w ) - (\nbg v)^T : \nbg w + \big(\nbg v \nu \nu^T\big) : \nbg w .
	\end{align}
\end{subequations}
See Lemma~\ref{lemma:interchange formulas} for proofs herein. We refer to \cite[Lemma~2.4 and 2.6]{DziukKronerMuller} for the original componentwise formulation.

For the remainder of the paper we shall assume that the velocity vector $v$ points in normal direction, i.e.,
\begin{equation*}
	v=V\nu ,
\end{equation*}
where $V$ is the normal velocity of $\Gamma(t)$. 

In this setting we have the following basic identities which can be found in  \cite{Huisken1984,Mantegazza,BGN_survey}:
\begin{subequations}
	\label{eq:three identities}
	\begin{align}
		\label{eq:three identities - mat nu}
		\mat \nu = &\ - \nbg V , \\
		\label{eq:three identities - mat H}
		\mat H = &\ - \laplace_\Ga V - |A|^2 V . 
	\end{align}
\end{subequations}

In the next lemma we prove appropriate variants of \eqref{eq:three identities} which will be useful for the anisotropic case.

\begin{lemma}
\label{lemma:anisotropic three identities} 
	Let $\Ga[X]$ be a sufficiently smooth evolving hypersurface with velocity vector $v = V \nu$ and let
	$\gamma$ be an anisotropy function.  Then it holds
	\begin{subequations}
		\label{eq:three identities-aniso}
		\begin{align}
			\mat \ga'(\nu) = &\ - \ga''(\nu) \nbg V ,
			\label{eq:three identities - normal-aniso}\\
			\label{eq:three identities - mat H-aniso}
			\mat H_\ga = &\ 
			- \nbg \cdot \Big( \ga''(\nu) \nbg V \Big) - |A|_{\ga''}^2 V.
		\end{align}
	\end{subequations}
\end{lemma}
\begin{proof} 
Equation \eqref{eq:three identities - normal-aniso} follows directly from the chain rule and 
\eqref{eq:three identities - mat nu}, i.e.,
\begin{align*}
	\mat (\ga'(\nu)) = &\ \ga''(\nu) \mat \nu = - \ga''(\nu) \nbg V .
\end{align*}

Next, using the interchange formula \eqref{eq:interchange formula - mat-div} and the definition of $H_\ga$, we directly compute
\begin{equation}
	\begin{aligned}
		\mat H_\ga = &\ \mat \big( \nbg \cdot ( \ga'(\nu) ) \big) \\
		= &\  \nbg \cdot \big( \mat (\ga'(\nu)) \big)
		- (\nbg v)^T : \nbg ( \ga'(\nu) ) + \big( \nbg v \nu \nu^T \big) : \nbg ( \ga'(\nu) ) . 
	\end{aligned}
\end{equation}
We now compute all three terms of the right-hand side more precisely.

In the first term on the right-hand side, we use \eqref{eq:three identities - normal-aniso} to obtain
$$ \nbg \cdot \big( \mat (\ga'(\nu)) \big)=-\nbg \cdot \Big( \ga''(\nu) \nbg V \Big).$$


In the second term we eliminate $v = V \nu$ 
and use the identity
\begin{equation*}
	\nbg v = \nbg (V \nu) = \nbg V \otimes \nu + V \nbg \nu , 
\end{equation*}
to obtain
\begin{align*}
	(\nbg v)^T : \nbg ( \ga'(\nu) )
	= &\ \tr \Big( \big( \nbg V \otimes \nu + V \nbg \nu \big) \nbg \nu \ga''(\nu) \Big) \\
	= &\ \tr \Big( \nbg V \nu^T \nbg \nu \ga''(\nu) + V \nbg \nu \nbg \nu \ga''(\nu) \Big) \\
	= &\ \nbg \nu : \big( \ga''(\nu) \nbg \nu \big) V ,
\end{align*}
where the last equality holds because both  $\nbg \nu$ and $ \ga''(\nu) $ are symmetric and since $\nu^T \nbg \nu = 0$.

The third term vanishes. Indeed, by using
the symmetry of   $\nbg \nu$ and $ \ga''(\nu) $,
we have
\begin{align*}
	\big( \nbg v \nu \nu^T \big) : \nbg ( \ga'(\nu) ) 
	= &\ \tr \Big( \nbg v \nu \nu^T \ga''(\nu) \nbg \nu \Big) = 0 ,
\end{align*}
where the last equality is a consequence of \eqref{eq:aniso2}. 
\qed
\end{proof}

\section{Anisotropic mean curvature flow}
\label{section:anisotropic surface flows}

\subsection{\it Evolution law} 
Let $\gamma$ be an anisotropy function as introduced in Section \ref{sec:anisotropy}. 
We say that a family $(\Gamma(t))_{t \in [0,T]}$ of closed, embedded hypersurfaces evolves by anisotropic mean curvature flow if 
\begin{equation}
	\label{eq:anisotropicMCF - beta}
	\beta(\nu) V = - H_\ga \qquad \mbox{ on } \Gamma(t).
\end{equation}
Here, $V$ is the normal velocity of $\Gamma(t)$ and  $\beta \in C^2( \R^{d+1}, \mathbb R_{>0})$ is a given kinetic function. Under this assumption 
there exist $c_2>0, c_3 \geq 0$ such that
\begin{eqnarray}
c_2 \leq \beta(v) & \leq & c_3 \qquad \qquad \; \, \forall v \in \mathbb R^{d+1}, \frac{1}{2}  \leq |v| \leq 2, \label{eq:betacond1} \\
| \beta(v_1) - \beta(v_2) |&  \leq &  c_3 | v_1 - v_2| \quad \forall  \,  v_1,v_2 \in \mathbb R^{d+1}, \, \frac{1}{2} \leq  |v_1|,|v_2| \leq 2,   \label{eq:betacond2} \\
| \beta'(v_1) - \beta'(v_2) |&  \leq &  c_3 | v_1 - v_2| \quad \forall  \,  v_1,v_2 \in \mathbb R^{d+1}, \, \frac{1}{2} \leq  |v_1|,|v_2| \leq 2.  \label{eq:betacond3}
\end{eqnarray}

For the specific choice $\beta \equiv 1$ we obtain
\begin{equation}
	\label{eq:anisotropicMCF}
	V = - H_\ga \qquad \mbox{ on } \Gamma(t).
\end{equation}
In the isotropic case,  \eqref{eq:anisotropicMCF} is the classical mean curvature flow, i.e.,~$V = - H$.

In what follows we assume a parametric description of the hypersurfaces $\Gamma(t)$ as outlined in Section \ref{sec:evolving}, so that 
$\Gamma(t)= \lbrace X(p,t) \, : \, p \in \Gamma^0 \rbrace$ for some smooth mapping $X \colon \Gamma^0 \times[0,T] \rightarrow \mathbb R^{d+1}$. Here, $\Gamma^0$ is the given initial hypersurface. In
order to satisfy \eqref{eq:anisotropicMCF - beta} taking initial conditions into account, we require
\begin{subequations}
\label{eq:anisotropicMCF - problem}
	\begin{align}
		\label{eq:surface evolution} 
		\partial_t X = &\ v \circ X \qquad \mbox{ on } \Gamma^0 \times (0,T), \quad \mbox{ where } v = V \nu \,,\\
		\ 
		 \beta(\nu) V =& - H_\gamma = - \nbg \cdot \big( \ga'(\nu) \big) , \\
		\label{eq:initcond}
		X(\cdot,0) = &\ \mbox{id}_{\Gamma^0} \qquad  \; \mbox{ on } \Gamma^0. 
	\end{align}
\end{subequations}

In order to set up our numerical method we follow the approach introduced in \cite{MCF} and discretize a system that involves not only the position vector $X$, but also the normal vector $\nu$  and the normal velocity $V$.
The corresponding evolution equations will be derived in the next section. 



\subsection{\it Evolution equations for anisotropic mean curvature flow}


\begin{lemma}
	\label{lemma:aMCF evolution equations - beta}
	Let $\Ga[X]$ be a sufficiently smooth solution of the anisotropic mean curvature flow  \eqref{eq:anisotropicMCF - problem}. Then the normal vector and normal velocity satisfy
	\begin{subequations}
		\label{eq:evolution eq - modified aMCF}
		\begin{align}
			\beta(\nu) \mat \nu 
			= &\ \nbg \cdot \Big( \ga''(\nu) \nbg \nu \Big) + |A|_{\ga''}^2 \, \nu + V \nbg \nu \beta'(\nu) ,\label{eq:evolution eq - modified aMCF(a)} \\ 
			\beta(\nu) \mat V 
			= &\ \nbg \cdot \Big( \ga''(\nu) \nbg V \Big) + |A|_{\ga''}^2 V  +  V \nbg V \cdot \beta'(\nu) .
		\end{align}
	\end{subequations}
\end{lemma}
\begin{proof}
	(a) 
	 Using \eqref{eq:three identities - mat nu}, $\beta(\nu) V = -H_\gamma$, the product rule and the chain rule yields
	\begin{align*}
		\beta(\nu) \mat \nu = & - \beta(\nu) \nbg V 
		= - \nbg \big( \beta(\nu) V \big) + V\nbg \big(\beta(\nu)\big)  =\nbg H_\ga + V \nbg \nu \beta'(\nu) \\
		=& \quad  \nbg \cdot \Big( \ga''(\nu) \nbg \nu \Big) + |A|_{\ga''}^2 \nu + V \nbg \nu \beta'(\nu)
	\end{align*}
	in view of  \eqref{eq:three identities - H and nu-aniso}, which gives \eqref{eq:evolution eq - modified aMCF(a)}.

	
	(b) 
	We use $\beta(\nu) V = -H_\gamma$, the product rule and the chain rule to obtain
	\begin{align*}
		\beta(\nu) \mat V = &\  \mat (\beta(\nu)V )- \mat \big( \beta(\nu) \big) V = - \mat H_\ga - V \beta'(\nu) \cdot \mat \nu  \\
		= & \  \nbg \cdot \Big( \ga''(\nu) \nbg V \Big) +  |A|_{\ga''}^2 V +  V \nbg V \cdot \beta'(\nu) 
	\end{align*}
	
recalling  \eqref{eq:three identities - mat H-aniso} and  \eqref{eq:three identities - mat nu}. 
	\qed
\end{proof}

Because the case $\beta \equiv 1$ is particularly relevant we state the corresponding evolution equations separately:
\begin{corollary}
	\label{lemma:aMCF evolution equations}
	Let $\Ga[X]$ be a sufficiently smooth solution of the anisotropic mean curvature flow \eqref{eq:anisotropicMCF - problem} with $V=-H_\gamma$. Then the normal vector and normal velocity satisfy
	\begin{subequations}
		\label{eq:evolution eq - aMCF}
		\begin{align}
			\mat \nu = &\ \nbg \cdot \Big( \ga''(\nu) \nbg \nu \Big) + |A|_{\ga''}^2 \, \nu , \label{eq:evolution eq - aMCF(a)}\\		
			\mat V = &\ \nbg \cdot \Big( \ga''(\nu) \nbg V \Big) + |A|_{\ga''}^2 V .  \label{eq:evolution eq - aMCF(b)}
		\end{align}
	\end{subequations}
\end{corollary}

\begin{remark}
\label{remark:isotropic MCF}
	In the isotropic case Lemma \ref{lemma:aMCF evolution equations} simplifies to the evolution equations of Huisken \cite{Huisken1984}:
	\begin{align*}
		\mat \nu = &\ \laplace_\Ga \nu  + |A|^2 \, \nu , \\
		\mat H = &\ \laplace_\Ga H + |A|^2 H .
	\end{align*}
\end{remark}

\begin{remark}
	Note that since we have $|A|_{\ga''}^2 = \nbg \nu \cdot ( \ga''(\nu) \nbg \nu ) \geq 0$, via the maximum principle for \eqref{eq:evolution eq - aMCF(b)} we conclude from $H_\ga(\cdot,0) \geq 0$ that $H_\ga(\cdot,t) \geq 0$ for all $t > 0$.
\end{remark}

\subsection{\it Weak formulation of the coupled system}
\label{section:weak formulations}

The weak formulation of the coupled system for anisotropic mean curvature flow \eqref{eq:anisotropicMCF - problem} with \eqref{eq:evolution eq - modified aMCF} reads: Given an  initial surface $\Ga^0$, find functions
$X \colon \Gamma^0 \times [0,T] \rightarrow \mathbb R^{d+1}$, $\nu \colon G_T \rightarrow \mathbb R^{d+1}$, $V \colon G_T \rightarrow \mathbb R$ such that
\begin{subequations}
	\label{eq:coupled system - weak - modified aMCF}
	\begin{align}
	         \label{eq:coupled system - weak - modified aMCF - ODE}
		&\  \partial_t X = v \circ X , \\
		\label{eq:coupled system - weak - modified aMCF - velocity law}
		&\ v = V \nu , \\
		\label{eq:coupled system - weak - modified aMCF - nu}
		&\ \int_{\Ga[X]} \!\!\! \beta(\nu) \, \mat \nu  \cdot \vphi^\nu + \int_{\Ga[X]} \!\!\! \ga''(\nu) \nbg \nu : \nbg \vphi^\nu \nonumber \\
		&\ \qquad\qquad\qquad = \int_{\Ga[X]} \!\!\! |A|_{\ga''}^2 \, \nu \cdot \vphi^\nu 
		+ \int_{\Ga[X]} \!\!\! V \nbg \nu \beta'(\nu) \cdot \vphi^\nu, \\
		\label{eq:coupled system - weak - modified aMCF - V}
		&\ \int_{\Ga[X]} \!\!\! \beta(\nu) \, \mat V \vphi^V + \int_{\Ga[X]} \!\!\! \ga''(\nu) \nbg V \cdot \nbg \vphi^V \nonumber \\
		&\ \qquad\qquad\qquad = \int_{\Ga[X]} \!\!\! |A|_{\ga''}^2 V \vphi^V 
		+ \int_{\Ga[X]} \!\!\! V \nbg V \cdot \beta'(\nu) \vphi^V, 
		\end{align}
\end{subequations}
for all $\vphi^\nu \in H^1(\Ga[X])^{d+1}$ and $\vphi^V \in H^1(\Ga[X])$.
The system is endowed with geometrically compatible initial data: $X(\cdot,0) = \mbox{id}_{\Gamma^0}$, $\nu(\cdot,0) = \nu^0$ being the outward normal vector field of $\Ga^0$, and $V(\cdot,0) = V^0 = - \tfrac{1}{\beta(\nu^0)} H_\ga^0 = - \tfrac{1}{\beta(\nu^0)} \nbg \cdot \big( \ga'(\nu^0) \big)$.


\section{Evolving surface finite element discretization}
\label{section:ESFEM}

For the spatial semi-discretization of the system  \eqref{eq:coupled system - weak - modified aMCF} we will use the evolving surface finite element method (ESFEM) \cite{Dziuk88,DziukElliott_ESFEM}.  We use curved simplicial finite elements and basis functions defined by continuous piecewise polynomial basis functions of degree~$k$ on triangulations, as defined in \cite[Section~2]{Demlow2009}, \cite{highorderESFEM}, and \cite{EllRan21}. 
The description below is almost verbatim to \cite{KLLP2017,MCF,MCFdiff}.

\subsection{\it Evolving surface finite elements}
The given smooth initial surface $\Ga^0$ is triangulated by an admissible family of triangulations $\mathcal{T}_h$ of decreasing maximal element diameter $h$; see \cite{DziukElliott_ESFEM,EllRan21} for the notion of an admissible triangulation, which includes quasi-uniformity and shape regularity. For a momentarily fixed $h$, $\bfx^0$ denotes the vector in $\R^{(d+1)\dof}$ that collects all nodes $p_j$ $(j=1,\dots,\dof)$ of the initial triangulation. By piecewise polynomial interpolation of degree $k$, the nodal vector defines an approximate surface $\Ga_h^0$ that interpolates $\Ga^0$ in the nodes $p_j$. Associated with these nodes are piecewise polynomial
basis function $\phi_1,\ldots,\phi_N$. \\
Let us next fix a  mapping $\bfx \colon [0,T] \to \R^{(d+1)\dof}$, such that $\bfx(0) = \bfx^0$. We associate with $\bfx$ the family of  discrete surfaces
$$
\Ga_h[\bfx(t)]=\Ga[X_h(\cdot,t)] = \{ X_h(p_h,t) \,:\, p_h \in \Ga_h^0 \}, \quad 0 \leq t \leq T
$$
parametrized by  the maps $X_h(p_h,t) = \sum_{j=1}^\dof x_j(t) \, \phi_j(p_h), \, p_h \in \Ga_h^0$. 
We abbreviate $\bfx(t)$ to $\bfx$ when the dependence on $t$ is clear from the context. \\
For $t \in [0,T]$ we  introduce globally continuous finite element \emph{basis functions} $\phi_i[\bfx(t)] \colon \Ga_h[\bfx(t)]\to\R, i=1,\dotsc,\dof$, 
such that on every element their pullback to the reference triangle is a polynomial of degree $k$, which satisfies $\phi_i[\bfx(t)](x_j(t)) = \delta_{ij}$ for all $i,j = 1, \dotsc, \dof$.
These functions span the finite element space 
\begin{equation*}
	S_h[\bfx(t)] = S_h(\Ga_h[\bfx(t)])=\spn\big\{ \phi_1[\bfx(t)], \phi_2[\bfx(t)], \dotsc, \phi_\dof[\bfx(t)] \big\} 
\end{equation*}
and satisfy the following transport property, see  \cite{DziukElliott_ESFEM},
\begin{equation} \label{eq:transport}
\frac\d{\d t} \big( \phi_i[\bfx(t)](X_h(p_h,t)) \big) =0 \qquad \forall p_h \in \Ga_h^0 .
\end{equation}
For a finite element function $u_h\in S_h[\bfx(t)]$, the tangential gradient $\nabla_{\Ga_h[\bfx(t)]}u_h$ is defined piecewise on each element. Furthermore,
we define the discrete velocity  $v_h(\cdot,t) \colon \Ga_h[\bfx(t)]  \to  \R^{d+1}$ by the relation
$$
 \partial_t X_h(p_h,t) = v_h(X_h(p_h,t),t). 
$$
In view of \eqref{eq:transport} we have $v_h(x,t) = \sum_{j=1}^\dof \dot x_j(t) \, \phi_j[\bfx(t)](x), x \in \Ga_h[\bfx(t)]$, where the dot denotes the time derivative $\d/\d t$. In particular,
the discrete velocity $v_h(\cdot,t)$ is in the finite element space $S_h[\bfx(t)]$, with nodal vector $\bfv(t)=\dot\bfx(t)$. Finally, the \emph{discrete material derivative} of a finite element function $u_h(x,t)$ with nodal values $u_j(t)$ is
$$
\mat_h u_h(x,t) = \frac{\d}{\d t} u_h(X_h(p_h,t),t) = \sum_{j=1}^\dof \dot u_j(t)  \phi_j[\bfx(t)](x)  \quad\text{at}\quad x=X_h(p_h,t).
$$
Suppose that $X \colon \Gamma^0 \times [0,T] \rightarrow \mathbb R^{d+1}$ parametrizes a family of evolving surfaces $(\Gamma(t))_{t \in [0,T]}$. Defining $x^* \colon  [0,T] \to \R^{(d+1)N}, x_j^*(t):=X(p_j,t)$, we then call  $\Ga_h[\bfx^*(t)]=\Ga[X^*_h(\cdot,t)]$   the \emph{interpolating surface}
of $\Ga(t)$. Thus, $\Ga_h[\bfx^*(t)] = \{ X^*_h(p_h,t) \,:\, p_h \in \Ga_h^0 \}$, where   $X_h^*(\cdot,t)= \sum_{j=1}^\dof x^*_j(t) \, \phi_j$. In this case,  the discrete velocity is given by
\begin{equation}
\label{eq:definition exacts velo interp}
	v_h^*(x,t)= \sum_{j=1}^\dof v^*_j(t) \, \phi_j[\bfx^*(t)](x) \quad \mbox{ with } v^*_j(t)= \dot x^*_j(t)= \partial_t X(p_j,t).
\end{equation}
Note that   the evolving triangulated surface $\Gamma_h[\bfx^*(t)]$  associated with $X_h^*(\cdot,t)$ is   admissible for all $t\in [0,T]$, provided that  the flow map $X$ is sufficiently regular and
$h$ is sufficiently small.  

\subsection{\it Lifts and Ritz map}
\label{section:lifts}

Suppose again  that $X \colon \Gamma^0 \times [0,T] \rightarrow \mathbb R^{d+1}$  parametrizes a family of evolving surfaces $(\Gamma(t))_{t \in [0,T]}$ with corresponding interpolating surface $\Ga_h[\bfx^*]$. Let us denote  by $d(\cdot,t)$ the oriented distance function to $\Gamma(t)$ and
abbreviate $U(t):= \lbrace x \in \mathbb R^{d+1} \, : \, | d(x,t) | < \delta \rbrace$. It can be shown that 
for each  $x \in U(t)$ there exists a unique point    $x^\ell \in \Gamma[X(\cdot,t)]$ satisfying
\begin{equation} \label{eq:projection}
	x = x^\ell + d(x,t) \nu_{\Ga[X]}(x^\ell,t) 
\end{equation}
provided that $\delta$ is sufficiently small. For a function $w_h\colon  \Gamma_h[\bfx^*(t)] \rightarrow \mathbb R$ we define its lift $w_h^\ell\colon  \Gamma[X(\cdot,t)] \rightarrow \mathbb R$ by $w_h^{\ell}(x^\ell) = w_h(x)$, where $x$ and $x^{\ell}$ are connected via \eqref{eq:projection}. 
It is shown in \cite[Lemma~4.2]{DziukElliott_acta}, that there is a constant $c>0$ such that for all $h$ sufficiently small and all $t \in [0,T]$ we have the bounds
\begin{equation}
	\label{eq:norm equivalence}
	\begin{gathered}
		c^{-1} \|w_h\|_{L^2(\Ga_h[\xs])} \leq \|w_h^\ell\|_{L^2(\Ga[X])} \leq c\|w_h\|_{L^2(\Ga_h[\xs])} , \\
		c^{-1} \|w_h\|_{H^1(\Ga_h[\xs])} \leq \|w_h^\ell\|_{H^1(\Ga[X])} \leq c\|w_h\|_{H^1(\Ga_h[\xs])} .
	\end{gathered}
\end{equation}
The inverse lift assigns to a function $u\colon  \Gamma[X(\cdot,t)] \rightarrow \mathbb R$ the function $u^{-\ell} \colon \Ga_h[\xs\t] \to \R$ such that $u^{-\ell}(x)= u(x^\ell)$. It is shown in \cite[Proposition 2.3]{Demlow2009} that
\begin{equation} \label{eq:nudif}
| \nu^{-\ell} - \nu_{\Ga_h[\bfx^*]} | \leq c h^k,
\end{equation}
where $\nu_{\Ga_h[\bfx^*]}$ denotes the unit normal to the discrete surface $\Ga_h[\bfx^*]$. 
For more details on the lift $\,^\ell$, see \cite{Dziuk88,DziukElliott_acta,Demlow2009}. \\
Let us next introduce a Ritz map, which is adjusted to our anisotropic setting.  To begin, note that
\eqref{eq:convex1} together with \eqref{eq:nudif} imply that
\begin{displaymath}
w \cdot \ga''(\nu^{-\ell} ) w \geq \frac{c_0}{4} | w|^2 \qquad \mbox{ for all } w \in \mathbb R^{d+1}, w \cdot \nu_{\Ga_h[\bfx^*]}=0
\end{displaymath}
provided that $0< h \leq h_0$ and $h_0$ is small enough. This allows us to make the following definition.

\begin{definition} Given $u \in H^1(\Ga[X])$ we denote by $u_h^*  \in S_h[\xs]$ the unique solution of 
	\begin{equation}
		\label{eq:Ritzmap}
		\begin{aligned}
			&\ \int_{\Ga_h[\xs]} \!\!  u_h^* \, \vphi_h + \int_{\Ga_h[\xs]} \!\!\!\! \ga''(\nu^{-\ell}) \, \nb_{\Ga_h[\xs]} u^*_h \cdot \nb_{\Ga_h[\xs]} \vphi_h \\
			= &\ \int_{\Ga[X]} \!\!\!\! u \vphi_h^\ell + \int_{\Ga[X]} \!\!\!\! \ga''(\nu) \, \nb_{\Ga[X]} u \cdot \nb_{\Ga[X]} \vphi_h^\ell ,
		\end{aligned}
	\end{equation}
for all $\vphi_h \in S_h[\xs]$. 
\end{definition}
A similar Ritz map was  used in \cite[Definition~3.1]{KPower_quasilinear}, see also the variants \cite[Definition~6.1]{DziukElliott_L2}, \cite[Definition~6.1]{highorderESFEM}, and \cite[Definition~3.6]{EllRan21}.

It can be shown that (cf.~\cite[Theorem~3.1--3.2]{KPower_quasilinear} with a scalar non-linearity):
\begin{eqnarray}  
	\label{eq:ritzest}
	\Vert (u^*_h)^\ell - u \Vert_{L^2(\Ga[X])} + h \Vert (u^*_h)^\ell - u \Vert_{H^1(\Ga[X])} \leq C h^{k+1},  \\
	\Vert (\partial^\bullet_h u^*_h)^\ell - \partial^\bullet u \Vert_{L^2(\Ga[X])} + h \Vert (\partial^\bullet_h u^*_h)^\ell - \partial^\bullet u \Vert_{H^1(\Ga[X])} \leq C h^{k+1}.  \label{eq:ritzesttime}
\end{eqnarray}


\section{Discretization and main result}
\label{section:discretization and main result}

The evolving surface finite element spatial semi-discretization of the weak coupled parabolic system for the anisotropic mean curvature flow  \eqref{eq:coupled system - weak - modified aMCF} reads: 
Find the  nodal vector $\bfx(t)\in \R^{(d+1)\dof}$ and 
$\nu_h(\cdot,t)\in S_h[\bfx(t)]^{d+1}$, $V_h(\cdot,t)\in S_h[\bfx(t)]$ such that
\begin{subequations}
	\label{eq:semidiscretization - modified aMCF}
	\begin{align}
	        &\ 	\partial_t X_h(p_h,t) = v_h(X_h(p_h,t),t), \qquad p_h\in\Ga_h^0, \\
		&\ v_h = \Ih \Big( V_h \, \nu_h \Big) , \\
		\label{eq:semidiscretization - modified aMCF - nu}
		&\ \int_{\Ga_h[\bfx]} \!\!\! \beta(\nu_h) \mat_h \nu_h \cdot \vphi_h^\nu + \int_{\Ga_h[\bfx]} \!\!\gamma''(\nu_h)  \, \nb_{\Ga_h[\bfx]} \nu_h : \nb_{\Ga_h[\bfx]} \vphi_h^\nu \nonumber \\
		&\ \qquad\qquad\qquad = \int_{\Ga_h[\bfx]} \!\!\! |A_h|_{\gamma''_h}^2 \nu_h \cdot \vphi_h^\nu 
		+ \int_{\Ga_h[\bfx]} \!\!\! V_h \nb_{\Ga_h[\bfx]} \nu_h \beta'(\nu_h) \cdot \vphi_h^\nu , \\
		\label{eq:semidiscretization - modified aMCF - V}
		&\ \int_{\Ga_h[\bfx]} \!\!\! \beta(\nu_h)  \mat_h V_h \vphi_h^V + \int_{\Ga_h[\bfx]} \!\!\! \gamma''(\nu_h)  \nb_{\Ga_h[\bfx]} V_h \cdot \nb_{\Ga_h[\bfx]} \vphi_h^V \nonumber \\
		&\ \qquad\qquad\qquad = \int_{\Ga_h[\bfx]} \!\!\! |A_h|_{\gamma''_h}^2 V_h \vphi_h^V 
		+ \int_{\Ga_h[\bfx]} \!\!\! V_h \nb_{\Ga_h[\bfx]} V_h \cdot \beta'(\nu_h) \, \vphi_h^V , 
			\end{align}
\end{subequations}
for all $\vphi_h^V \in S_h[\bfx(t)]$ and $\vphi_h^\nu \in S_h[\bfx(t)]^{d+1}$. Here, we have abbreviated 
\begin{displaymath}
A_h = \half \Big(\nb_{\Ga_h[\bfx]} \nu_h + (\nb_{\Ga_h[\bfx]} \nu_h)^T\Big) \quad \mbox{ and } \quad |A_h|_{\ga''_h}^2 = A_h : \ga''(\nu_h)  A_h.
\end{displaymath}
 Furthermore, $\Ih f:= \sum_{j=1}^N f(x_j(t)) \phi_j[\bfx(t)]$ for a function $f: \Ga_h[\bfx(t)] \rightarrow \mathbb R$. Finally, the initial data  $\nu_h(\cdot,0)$, and $V_h(\cdot,0)$ are the Lagrange interpolations of  $\nu^0$, and $V^0 = - \tfrac{1}{\beta(\nu^0)} H_\ga^0 = - \tfrac{1}{\beta(\nu^0)} \nbg \cdot \big( \ga'(\nu^0) \big)$, respectively.

\medskip
\noindent
Our main result of this paper reads:

\begin{theorem}
	\label{theorem: error estimates - modified aMCF}
	Suppose that the  system \eqref{eq:coupled system - weak - modified aMCF}  admits a sufficiently smooth, regular solution $X \colon \Gamma^0 \times [0,T] \to \mathbb R^{d+1}$ ($d = 1,2,3$) with unit normal $\nu$ and normal velocity $V$. Consider
	the semi-discrete problem \eqref{eq:semidiscretization - modified aMCF}, discretized by evolving surface finite elements of polynomial degree $k \geq 2$. 
	Then there exists a constant $h_0 > 0$ such that for all mesh sizes $h \leq h_0$ \eqref{eq:semidiscretization - modified aMCF} has a unique solution $X_h \colon \Gamma^0_h \times [0,T] \to \mathbb R^{d+1}$, $\nu_h(\cdot,t)\in S_h[\bfx(t)]^{d+1}$, $V_h(\cdot,t)\in S_h[\bfx(t)]$ and the following error bounds hold:
	\begin{align*}
	 \max_{t \in [0,T]}    \|X_h^\ell(\cdot,t) - X(\cdot,t)\|_{H^1(\Ga^0)^{d+1}} \leq &\ Ch^k , \\
	\max_{t \in [0,T]}  \|\nu_h^L(\cdot,t) - \nu(\cdot,t)\|_{H^1(\Ga[X(\cdot,t)])^{d+1}} \leq &\ C h^k, \\
	\max_{t \in [0,T]} \|V_h^L(\cdot,t) - V(\cdot,t)\|_{H^1(\Ga[X(\cdot,t)])} \leq &\ C h^k. 
	\end{align*}
	Here, we have associated with a function  $w_h=\sum_{j=1}^N w_j \phi_j[\bfx(t)] \in  S_h[\bfx(t)] $   the function $w_h^L: \Gamma[X(\cdot,t)] \rightarrow \mathbb R$ via $w_h^L = (\widehat w_h)^\ell$, where $\widehat w_h=\sum_{j=1}^N w_j \phi_j[{\bfx ^*}(t)]  \in S_h[{\bfx^*} (t)]$, see \cite{MCF}.
	The constant $C$ is independent of $h$, but depends on bounds of higher derivatives of the solution $(X,\nu,V)$  and exponentially on $T$.
\end{theorem}

Using the theory of analytic semigroups (see, for example, \cite{Lunardi2013}) it is possible
to show that \eqref{eq:anisolaw} possesses, locally in time, a unique smooth
solution provided that $\gamma$, $\beta$ and the initial data are smooth enough. For proofs of the local-in-time existence of solutions of related problems in parabolic H\"older spaces, we also refer the reader to \cite{GigaG1992,ATW1993,MikulaS2001}, and to the discussion at the end of Chapter 1 in \cite{Giga_2006}.

Regularity assumptions  for Theorem~\ref{theorem: error estimates - modified aMCF} are the following: We assume $X(\cdot,t) \in  H^{k+1}(\Ga^0)$,
$v(\cdot,t) \in H^{k+1}(\Ga(X(\cdot,t)))$, and for $u=(\nu,V)$ we require $u(\cdot,t)$, $\mat u(\cdot,t) \in W^{k+1,\infty}(\Ga(X(\cdot,t)))^{d+2}$ with bounds that are uniform in $t\in[0,T]$.

\section{Error analysis}
\label{sec:erroranalysis}

\subsection{\it Matrix--vector formulation}
\label{section:matrix_vector}

Let us write $X_h(\cdot,t) = \sum_{j=1}^\dof x_j(t) \, \phi_j$ with  $\bfx(t) = (x_j(t))_{j=1}^N$ and collect the nodal values of   $\nu_h(\cdot,t)\in S_h[\bfx(t)]^{d+1}$, $V_h(\cdot,t)\in S_h[\bfx(t)]$ in the column
vectors $\bfn=(n_j)_{j=1}^N, \bfV=(V_j)_{j=1}^N$. Furthermore, we set
\begin{displaymath}
	\bfu := \begin{pmatrix} \bfn \\ \bfV \end{pmatrix} \in \R^{(d+2)\dof} , \quad \mbox{ as well as } \quad  \bfV \bullet \bfn = (V_j n_j)_{j=1}^N.
\end{displaymath}
We may then express \eqref{eq:semidiscretization - modified aMCF}  in matrix-vector form as
\begin{subequations}
\label{eq:matrix-vector form - modified aMCF}
	\begin{align}
	\label{eq:matrix-vector form - modified aMCF - c}
		\dot \bfx = &\  \bfv  , \\
		\label{eq:matrix-vector form - modified aMCF - a}
		\bfv = &\ \bfV \bullet \bfn , \\
		\label{eq:matrix-vector form - modified aMCF - b}
		\bfM^{[d+2]}(\bfx,\bfu) \dot{\bfu} + \bfA^{[d+2]}(\bfx,\bfu) \bfu = &\  \bffb(\bfx,\bfu) ,
	\end{align} 
\end{subequations}
where the solution-dependent mass  and stiffness matrices are given by
\begin{align*}
	\bfM(\bfx,\bfu)|_{ij} &= \int_{\Ga_h[\bfx]} \!\!\! \beta(\nu_h) \, \phi_i[\bfx] \phi_j[\bfx] , \\
	\bfA(\bfx,\bfu)|_{ij} &= \int_{\Ga_h[\bfx]} \!\!\! \gamma''(\nu_h) \, \nbgh \phi_i[\bfx]  \cdot \nbgh \phi_j[\bfx] 
\end{align*}
for $i,j=1,\ldots,N$, and we have abbreviated 
\begin{displaymath}
(\bfM^{[d+2]},\bfA^{[d+2]})(\bfx,\bfu)= I_{d+2} \otimes (\bfM,\bfA)(\bfx,\bfu)
\end{displaymath}
where $I_{d+2}$ is the identity matrix. Also  $\bff(\bfx,\bfu) = \big( \bff_1(\bfx,\bfu) , \bff_2(\bfx,\bfu) \big)^T$ with
\begin{align*}
	\bffb_1(\bfx,\bfu)|_{j + (\ell - 1)N} 
	= &\  \int_{\Ga_h[\bfx]} \!\!\! |A_h|_{\gamma''_h}^2 \nu_{h,\ell} \, \phi_j[\bfx] 
	+ \int_{\Ga_h[\bfx]} \!\!\! V_h (\nb_{\Ga_h[\bfx]} \nu_h \beta'(\nu_h))_\ell \, \phi_j[\bfx] , \\
	\bffb_2(\bfx,\bfu)|_{j } 
	= &\  \int_{\Ga_h[\bfx]} \!\!\! |A_h|_{\gamma''_h}^2 V_h \, \phi_j[\bfx] 
	+ \int_{\Ga_h[\bfx]} \!\!\! V_h \nb_{\Ga_h[\bfx]} V_h \cdot \beta'(\nu_h) \, \phi_j[\bfx] , 
\end{align*}
for $j = 1, \dotsc, \dof$ and $\ell=1,\dotsc,d+1$. 

Note that the above matrix--vector formulation \eqref{eq:matrix-vector form - modified aMCF} is particularly close to those for: mean curvature flow \cite[equation~(3.4)]{MCF} (basic structure), for generalised mean curvature flow \cite[equation~(3.4)]{MCF_generalised} (solution-dependent mass matrix $\bfM(\bfx,\bfu)$), and diffusion coupled to mean curvature flow \cite[equation~(6.3)]{MCFdiff} (solution-dependent stiffness matrix $\bfA(\bfx,\bfu)$). We will exploit these similarities throughout the numerical analysis of this method.

\subsection{\it Consistency bounds}

Let us define $X_h^*(\cdot,t)= \sum_{j=1}^\dof x^*_j(t) \, \phi_j$, where $x^*_j(t) = X(p_j,t)$. Furthermore, we denote by $\nu^*_h \in S_h[\xs]^{d+1}, V_h^* \in S_h[\xs]$ the Ritz maps \eqref{eq:Ritzmap} of $\nu,V$, whose nodal
values are collected in $\bfn^*$ and $\bfV^*$ respectively. As above we write
\begin{displaymath}
	 \bfu^*= \begin{pmatrix} \bfn^* \\ \bfV^* \end{pmatrix} \in \R^{(d+2)\dof}, \quad \mbox{ as well as } \quad \bfV^* \bullet \bfn^* = (V^*_j n^*_j)_{j=1}^N
\end{displaymath}
and observe that  $(\bfx^*,\bfn^*,\bfV^*)$ satisfies the system 
\begin{subequations}
	\label{eq:defect definition - modified aMCF}
	\begin{align}
	    \dotxs = &\ \vs , \\
		\vs =&\ \bfV^* \bullet \bfn^* + \dv , \\
		\bfM^{[d+2]}(\xs,\us) \dotus + \bfA^{[d+2]}(\xs,\us) \us = &\ \bffb(\xs,\us) + \bfM(\xs)\du . \label{eq:defect definition - modified aMCF-c}
	\end{align}
\end{subequations}
for suitably defined \emph{defects} $d_v  = \sum_{j=1}^N d_{v,j} \phi_j[\bfx^*], d_u  = \sum_{j=1}^N d_{u,j} \phi_j[\bfx^*]$ and
with the standard mass matrix $\displaystyle \bfM(\xs)|_{ij} = \int_{\Gamma_h[\xs]} \phi_i[\xs] \phi_j[\xs]$.

\begin{lemma} We have
	\label{lemma:consistency - modified aMCF}
	\begin{align}
	\max_{t \in [0,T]} 	\|d_v(\cdot,t)\|_{H^1(\Ga_h[\xs(t)])} + \max_{t \in [0,T]} 	\|d_u(\cdot,t)\|_{L^2(\Ga_h[\xs(t)])}  \leq \ C h^k . \label{eq:dvuest} 
	\end{align}
\end{lemma}
\begin{proof} Let us write $d_v = v^*_h - \tilde I_h^* (V_h^* \nu^*_h)$ with $v_h^* = \tilde I_h^* (v^{-\ell}) = \tilde I_h^* (V^{-\ell} \nu^{-\ell})$, where   $\tilde I_h^* f= \sum_{j=1}^N f(x^*_j(t)) \phi_j[\bfx^*(t)]$. It is shown in \cite[Lemma 5.3]{KLL_Willmore} that
\begin{displaymath}
\Vert \tilde I_h^* (a_h b_h ) \Vert_{H^1(\Ga_h[\xs(t)])}  \leq C \Vert a_h \Vert _{H^1(\Ga_h[\xs(t)])} \Vert b_h \Vert_{W^{1,\infty}(\Ga_h[\xs(t)])}, \; a_h, b_h \in S_h[\bfx^*(t)].
\end{displaymath}
Combining this bound with \eqref{eq:ritzest} and the interpolation estimates from  \cite[Proposition~2.7]{Demlow2009}, and recalling \eqref{eq:definition exacts velo interp}, we deduce that
\begin{eqnarray*}
	\Vert d_v \Vert_{H^1(\Ga_h[\xs(t)])} 
	& = & \Vert \tilde I_h^* (V^{-\ell} \nu^{-\ell}) - \tilde I_h^* (V_h^* \nu^*_h) \Vert _{H^1(\Ga_h[\xs(t)])}  \\
	& \leq & \Vert \tilde I_h^* \{ \tilde I_h^* ( V^{-\ell} - V^*_h) \tilde I_h^* \nu^{-\ell} \} \Vert _{H^1(\Ga_h[\xs(t)])} 
	\\&&\qquad \quad +\Vert \tilde I_h^* \{ V^*_h ( \tilde I_h^* \nu^{-\ell} - \nu^*_h) \} \Vert_{H^1(\Ga_h[\xs(t)])} \\
	& \leq & c \Vert \tilde I_h^* (V,\nu)^{-\ell} - (V^*_h,\nu^*_h) \Vert_{H^1(\Ga_h[\xs(t)])}  \\
	& \leq & c h^k,
\end{eqnarray*}
where we also used that $ \Vert \tilde I_h^* \nu^{-\ell} \Vert_{W^{1,\infty}(\Ga_h[\xs(t)])} + \Vert V^*_h \Vert_{W^{1,\infty}(\Ga_h[\xs(t)])}  \leq C$. 

Next, we obtain by writing \eqref{eq:defect definition - modified aMCF-c} in terms of $d_u$ 
and using \eqref{eq:coupled system - weak - modified aMCF - nu}, \eqref{eq:coupled system - weak - modified aMCF - V} as well as \eqref{eq:Ritzmap} for $\varphi_h \in S_h[\bfx^*]^{d+2}$
\begin{eqnarray*}
\lefteqn{ 
\int_{\Ga_h[\bfx^*]} d_u \cdot \varphi_h } \\
& = &  \int_{\Ga_h[\bfx^*]} \beta(\nu^*_h)\partial^\bullet_h u^*_h \cdot \varphi_h + \int_{\Ga_h[\bfx^*]} \gamma''(\nu^*_h) \nabla_{\Ga_h[\bfx^*]} u^*_h : \nabla_{\Ga_h[\bfx^*]} \varphi_h   \\
& & - \int_{\Ga_h[\bfx^*]} | A_h^* |^2_{\gamma_h''} u^*_h \cdot \varphi_h - \int_{\Ga_h[x^*]}  V_h^* (\nabla_{\Ga_h[x^*]} u^*_h)^T \,  \beta'(\nu_h^*) \cdot \varphi_h \\
& = & \Bigl(  \int_{\Ga_h[\bfx^*]} \beta(\nu^*_h)\partial^\bullet_h u^*_h \cdot \varphi_h - \int_{\Gamma[X]} \beta(\nu)\partial^\bullet u \cdot \varphi_h^{-\ell} \Bigr) \\
& & - \Bigl( \int_{\Ga_h[\bfx^*]}  u^*_h \cdot \varphi_h - \int_{\Gamma[X]}  u \cdot \varphi_h^{-\ell} \Bigr)   \\
& & + \int_{\Ga_h[\bfx^*]} \bigl( \gamma''(\nu_h^*) - \gamma''(\nu^{-\ell}) \bigr) \nabla_{\Ga_h[\bfx^*]} u^*_h : \nabla_{\Ga_h[\bfx^*]} \varphi_h \\
& & - \Bigl( \int_{\Ga_h[\bfx^*]} | A_h^* |^2_{\gamma_h''} u^*_h \cdot \varphi_h - \int_{\Ga[X]} | A|_{\gamma''}^2 u \cdot \varphi_h^{-\ell} \Bigr) \\
& & - \Bigl( \int_{\Ga_h[\bfx^*]} V_h^* (\nabla_{\Ga_h[x^*]} u^*_h)^T \,  \beta'(\nu_h^*) \cdot \varphi_h - \int_{\Ga[X]} V (\nabla_{\Gamma[X]} u)^T \beta'(\nu) \cdot \varphi_h^{-\ell} \Bigr) \\
& = &: J_1(\varphi_h) +\ldots +J_5(\varphi_h).
\end{eqnarray*}
By transforming the integrals over $\Gamma_h[\bfx^*]$ to $\Gamma[X]$ and using \eqref{eq:ritzest}, \eqref{eq:ritzesttime}  and \eqref{eq:betacond2} one shows that
\begin{displaymath}
| J_1(\varphi_h)|  + |J_2(\varphi_h)|  \leq C h^{k+1} \Vert \varphi_h \Vert_{L^2(\Ga_h[\bfx^*])}.
\end{displaymath}
Next, we infer from \eqref{eq:gamma2}, \eqref{eq:ritzest} and an inverse estimate that 
\begin{eqnarray*}
| J_3(\varphi_h) | & \leq &  C_1 \int_{\Ga_h[\bfx^*]} | \nu^*_h - \nu^{-\ell}| \,  | \nabla_{\Ga_h[\bfx^*]} u^*_h  | \, |  \nabla_{\Ga_h[\bfx^*]} \varphi_h | \\
& \leq & C \Vert u^*_h - u^{-\ell} \Vert_{L^2(\Ga_h[\bfx^*])} \Vert \nabla_{\Ga_h[\bfx^*]} \varphi_h \Vert_{L^2(\Ga_h[\bfx^*])} \\
& \leq & C h^{k+1} \Vert \nabla_{\Ga_h[\bfx^*]} \varphi_h \Vert_{L^2(\Ga_h[\bfx^*])}  \leq C h^k \Vert \varphi_h \Vert_{L^2(\Ga_h[\bfx^*])} .
\end{eqnarray*}
Transforming again to $\Ga[X]$ and using \eqref{eq:gamma2}, \eqref{eq:betacond3} as well as
\begin{displaymath}
	\big| \,  | A_h^*|^2_{\gamma''_h} - | A |^{2,-\ell}_{\gamma''} \big| \leq C \bigl( | u^*_h - u^{-\ell} | + | \nabla_{\Ga_h[\bfx^*]}( u^*_h - u^{-\ell}) | \bigr)
\end{displaymath}
we infer that
\begin{displaymath}
	| J_4(\varphi_h) | + | J_5(\varphi_h) | \leq C \Vert u^*_h - u^{-\ell} \Vert_{H^1(\Ga_h[\bfx^*])} \Vert \varphi_h \Vert_{L^2(\Ga_h[\bfx^*])}  \leq C h^k \Vert \varphi_h \Vert_{L^2(\Ga_h[\bfx^*])} .
\end{displaymath}
Combining the above bounds we deduce that $| \int_{\Gamma_h[\bfx^*]} d_u \cdot \varphi_h | \leq C h^k \Vert\varphi_h \Vert_{L^2(\Ga_h[\bfx^*])}$ for all $\varphi_h \in S_h[\bfx^*]^{d+2}$ which yields \eqref{eq:dvuest}. \qed
\end{proof}
\noindent
For later use we observe that \eqref{eq:ritzest}, \eqref{eq:nudif} along with an inverse estimate imply that
\begin{equation} \label{eq:normdif}
| \nu_h^* - \nu_{\Ga_h[\bfx^*]}  |  \leq   | \nu_h^* -  \nu^{-\ell} | + | \nu^{-\ell} -   \nu_{\Ga_h[\bfx^*]} | 
 \leq   c h^{k+1-\frac{d}{2}}+ c h^k \leq  \min\Big(\frac{\delta}{2},\frac{1}{4}\Big)
 \end{equation}
 if $0<h \leq h_0$, where $\delta$ is related to \eqref{eq:convex1}. 
 In particular, we have that  
 \begin{equation} \label{eq:nustar}
  \frac{3}{4} \leq | \nu_h^* | \leq \frac{3}{2} \quad \mbox{ a.e. on } \Gamma_h[\xs].
  \end{equation}

\subsection{\it Auxiliary results}
\label{section:auxiliary}

\noindent
Let us abbreviate
\begin{equation} \label{eq:defe}
\ex=\bfx - \bfx^*, \quad \eu=\bfu - \bfu^*, \quad \ev=\bfv - \bfv^*
\end{equation}
with corresponding finite element functions $e_x,e_u,e_v \in S_h[\bfx^*]$. Furthermore, we shall associate with a vector 
${\bf w} = (w_j)_{j=1}^N \in \mathbb R^{N}$ the finite element functions 
\begin{displaymath}
 w_h=\sum_{j=1}^N w_j \phi_j[\bfx], \quad w_h^*= \sum_{j=1}^N w_j \phi_j[\bfx^*].
\end{displaymath}

The following norm equivalence results were shown in \cite[Lemma~4.2, 4.3]{KLLP2017}.
\begin{lemma} \label{lem:normequiv}
Suppose that $	\| \nabla_{\Gamma_h[\bfx^*]} e_x \|_{L^\infty(\Gamma_h[\bfx^*])}\le \frac14$ and that $1 \leq p \leq \infty$. Then there exist $c_p, C_p>0$ such that
for all  $\bfw \in \mathbb R^{N}$
\begin{eqnarray}
c_p \, \Vert w^*_h \Vert_{L^p(\Gamma_h[\bfx^*])}
& \leq  & \Vert w_h \Vert_{L^p(\Gamma_h[\bfx])}   \leq   C_p \, \Vert w^*_h \Vert_{L^p(\Gamma_h[\bfx^*])} ; \label{eq:equivl2}  \\
 c_p \, \Vert  \nabla_{\Gamma_h} w^*_h  \Vert_{L^p(\Gamma_h[\bfx^*])} & \leq & 
 \Vert  \nabla_{\Gamma_h} w_h  \Vert_{L^p(\Gamma_h[\bfx])}   \leq   C_p\,  \Vert  \nabla_{\Gamma_h} w^*_h  \Vert_{L^p(\Gamma_h[\bfx^*])} . \label{eq:equivh1}
\end{eqnarray}
\end{lemma}

\begin{lemma}
	\label{lemma:solution-dependent mass matrix diff}
	There exists $\mu>0$ such that $\frac{1}{2} \leq | \nu_h | \leq 2$ and
	\begin{align}
		c  \Vert w^*_h \Vert_{L^2(\Gamma_h[\bfx^*])}^2  & \leq  \Vert \bfw \Vert^2_{\bfM(\bfx,\bfu)} \leq  C \Vert w^*_h \Vert_{L^2(\Gamma_h[\bfx^*])}^2;  \label{eq:Mxunormequiv} \\
		c  \Vert \nabla_{\Gamma_h} w^*_h \Vert_{L^2(\Gamma_h[\bfx^*])}^2  &  \leq \|\bfw\|_{\bfA(\bfx,\bfu)}^2 \leq  C \Vert \nabla_{\Gamma_h} w^*_h \Vert_{L^2(\Gamma_h[\bfx^*])}^2    \label{eq:Axunormequiv}
	\end{align}
	for all $\bfw \in \mathbb R^{N}$, provided that  $\Vert e_u \Vert_{L^\infty(\Gamma_h[\bfx^*])} \leq \mu, \Vert e_x \Vert_{W^{1,\infty}(\Gamma_h[\bfx^*])} \leq \mu$. Here,  we have abbreviated  $\Vert \bfw \Vert_{\bf S}^2= \bfw^T {\bf S} \bfw$ for a matrix ${\bf S} \in \mathbb R^{N \times N}$.
\end{lemma}
\begin{proof} Since $\frac{3}{4} \leq | \nu_h^* | \leq \frac{3}{2}$ there exists $\mu>0$ such that $\frac{1}{2} \leq | \nu_h | \leq 2$ provided that $\Vert e_u \Vert_{L^\infty(\Gamma_h[\bfx^*])} \leq \mu$. 
Observing that $\Vert \bfw \Vert^2_{\bfM(\bfx,\bfu)} = \int_{\Gamma_h[\bfx]} \beta(\nu_h) w_h^2$, \eqref{eq:Mxunormequiv} immediately follows  from \eqref{eq:betacond1} and \eqref{eq:equivl2}. 
Next, we have in view of \eqref{eq:normdif}
\begin{align*}
| \nu_h - \nu_{\Ga_h[\bfx]}  |  \leq &\ | \nu_h - \nu_h^*| +   | \nu_h^* -     \nu_{\Ga_h[\bfx^*]} |  + |  \nu_{\Ga_h[\bfx^*]}  -  \nu_{\Ga_h[\bfx]}|  \\
\leq  &\ \Vert e_u \Vert_{L^\infty(\Gamma_h[\bfx^*])}+  \frac{\delta}{2}  + c \Vert e_x \Vert_{W^{1,\infty}(\Gamma_h[\bfx^*])} < \delta ,
\end{align*}
if $\mu>0$ is small enough. Thus, \eqref{eq:convex1}  implies for all $v \in \mathbb R^{d+1}$ with $ v \cdot \nu_{\Ga_h[\bfx]}=0$
\begin{displaymath}
v \cdot \gamma''(\nu_h) v \geq \frac{c_0}{4} | v|^2 .
\end{displaymath}
Applying this estimate with $v=\nabla_{\Gamma_h} w_h$ and observing that
\begin{displaymath}
\bfw^T \bfA(\bfx,\bfu) \bfw =  \int_{\Ga_h[\bfx]} \!\!\! \gamma''(\nu_h) \, \nabla_{\Ga_h[\bfx]} w_h  \cdot \nabla_{\Ga_h[\bfx]} w_h ,
\end{displaymath}
we deduce the first bound in \eqref{eq:Axunormequiv} with the help of \eqref{eq:equivh1}.  The second bound follows from the  fact that  $| \gamma''(\nu_h)| \leq \max_{\frac{1}{2} \leq |v| \leq 2} | \gamma''(v)|$. 
 \qed
\end{proof}

\begin{lemma} Let  $\bfw,  \bfz \in \mathbb R^N$ and suppose that $\Vert z_h^* \Vert_{W^{1,\infty}(\Gamma_h[\xs])} \leq C$. Then
\begin{align}
\bfw^T \big( \bfM(\bfx,\bfu)-\bfM(\bfx,\us) \big) \bfz 
		&\leq C \, \| e_u \|_{L^2(\Gamma_h[\bfx^*])}  \, \| w_h^* \|_{L^2(\Gamma_h[\bfx^*])} ;  \label{eq:mdif1}\\
		\bfw^T \big( \bfM(\bfx,\bfu^*)-\bfM(\bfx^*,\bfu^*) \big) \bfz 
		&\leq C \, \| e_x \|_{H^1(\Gamma_h[\bfx^*])}  \| w_h^* \|_{L^2(\Gamma_h[\bfx^*])} ; \label{eq:mdif2} \\
		\bfw^T \big( \bfA(\bfx,\bfu) - \bfA(\bfx,\us) \big) \bfz 
		&\leq C \, \| e_u \|_{L^2(\Gamma_h[\bfx^*])}  \, \| w_h^* \|_{H^1(\Gamma_h[\bfx^*])}  ; \label{eq:adif1}\\
		\bfw^T \big( \bfA(\bfx,\bfu^*) - \bfA(\bfx^*,\us) \big) \bfz 
		&\leq C \, \| e_x \|_{H^1(\Gamma_h[\bfx^*])}    \| w_h^* \|_{H^1(\Gamma_h[\bfx^*])}  . \label{eq:adif2}
	\end{align}
\end{lemma}
\begin{proof} The estimates \eqref{eq:mdif1}, \eqref{eq:adif1} follow from \eqref{eq:betacond2}, \eqref{eq:gamma2} and Lemma \ref{lem:normequiv} together with the fact that $| \nu_h -\nu_h^* | \leq | u_h - u_h^* |$. In order to prove the remaining
bounds one can adapt Lemma 4.1 in \cite{KLLP2017} in a straightforward way and use the uniform boundedness of $\beta(\nu_h^*)$ and $\gamma''(\nu_h^*)$  respectively. \qed
\end{proof}
\noindent
In order to formulate our next result we recall that the discrete velocity is given by $v_h(\cdot,t)=\sum_{j=1}^N \dot{x}_j(t) \phi_j[\bfx(t)] \in S_h[\bfx(t)]^{d+1}$. 
\begin{lemma}
\label{lemma:quasi-linear stiffness matrix diff} 
	Let  $\bfw, \bfz \in \mathbb R^N$ and suppose that $\Vert z_h^* \Vert_{W^{1,\infty}(\Gamma_h[\xs])} \leq C$ as well as
	$\Vert  v_h \Vert_{W^{1,\infty}(\Gamma_h[\bfx])} \leq C$. Then
		\begin{align}
	 & 	\bfw^T \diff \bfA(\bfx,\bfu)   \bfz 
			\leq C \, \| w_h^* \|_{H^1(\Gamma_h[\bfx^*])}  \bigl( \| z_h^*  \|_{H^1(\Gamma_h[\bfx^*])} + \Vert \mat_h e_u  \Vert_{L^2(\Gamma_h[\bfx^*])}  \bigr)  ; \label{eq:atime1}\\
		&   \bfw^T \diff \big( \bfA(\bfx,\bfu) - \bfA(\bfx,\us) \big) \bfz   \leq C \, \bigl(  \| e_u \|_{L^2(\Gamma_h[\bfx^*])} + \Vert \mat_h e_u  \Vert_{L^2(\Gamma_h[\bfx^*])}  \bigr) \| w_h^* \|_{H^1(\Gamma_h[\bfx^*])}  ;  \label{eq:atime2}  \\
	&	\bfw^T \diff \big( \bfA(\bfx,\bfu^*) - \bfA(\bfx^*,\us) \big) \bfz 
			  \leq C \, \bigl(   \| e_x \|_{H^1(\Gamma_h[\bfx^*])}  + \Vert e_v \Vert_{H^1(\Gamma_h[\bfx^*])}  \bigr)  \| w_h^* \|_{H^1(\Gamma_h[\bfx^*])}  .  \label{eq:atime3}
		\end{align}
\end{lemma}
\begin{proof} In order to prove the first bound we use Lemma 5.2 in \cite{DziukElliott_acta} with $\mathcal A=\gamma''(\nu_h)$ and write
\begin{displaymath}
\bfw^T \diff \bfA(\bfx,\bfu)  \bfz  =  \diff \int_{\Gamma_h[\bfx]} \gamma''(\nu_h) \nabla_{\Gamma_h} w_h \cdot \nabla_{\Gamma_h} z_h = \int_{\Gamma_h[\bfx]} B(\nu_h,v_h) \nabla_{\Gamma_h} w_h \cdot \nabla_{\Gamma_h} z_h,
\end{displaymath}
where
\begin{equation} \label{eq:defB}
B(\nu_h,v_h):= \mat_h \gamma''(\nu_h)  + \tr(\nabla_{\Gamma_h} v_h) \gamma''(\nu_h) - (\gamma''(\nu_h) \nabla_{\Gamma_h} v_h + (\gamma''(\nu_h) \nabla_{\Gamma_h}v_h )^T) .
\end{equation}
Since $\mat_h \gamma''(\nu_h) = \gamma'''(\nu_h)\mat_h \nu_h$ and $| \mat_h \nu_h | \leq C \, (1+ | \mat_h e_u|)$ we deduce  
 \eqref{eq:atime1} using our assumptions on $\nabla_{\Gamma_h} v_h$  and $z_h^*$   as well as Lemma \ref{lem:normequiv}. 
Next, with the above
relation we have
\begin{displaymath}
\bfw^T \diff \big( \bfA(\bfx,\bfu) - \bfA(\bfx,\us) \big) \bfz = \int_{\Gamma_h[\bfx]} \bigl( B(\nu_h,v_h) - B(\nu_h^*,v_h) \bigr)  \nabla_{\Gamma_h} w_h \cdot \nabla_{\Gamma_h} z_h.
\end{displaymath}
Recalling \eqref{eq:defB} and using \eqref{eq:gamma2}, \eqref{eq:gamma3} we deduce that
\begin{eqnarray*}
| B(\nu_h,v_h) - B(\nu_h^*,v_h) | &  \leq &  | \gamma'''(\nu_h) \mat_h \nu_h - \gamma'''(\nu_h^*) \mat_h \nu_h^*| + c | \gamma''(\nu_h) - \gamma''(\nu_h^*) |  \\
& \leq & C \, \bigl( | e_u | + | \mat_h e_u | \bigr),
\end{eqnarray*}
from which we infer \eqref{eq:atime2}. 
The remaining estimate \eqref{eq:atime3} is obtained in a similar way as the bound (8.8) in \cite {MCFdiff}. \qed
\end{proof}

\subsection{\it Proof of the error bounds}
\label{section:proof of main theorems}

Let us define
\begin{align} 
\widehat{T}_h = & \sup \Big\{  t \in [0,T] :  \eqref{eq:semidiscretization - modified aMCF} \mbox{ has a solution on }
 [0,t],
\text{ and } 
\label{eq:hatTh} \\ 
& \hphantom{
	\sup \Big\{ 
	\, } \Vert (e_x(\cdot,s),e_v(\cdot,s), e_u(\cdot,s)) \Vert_{W^{1,\infty}([\Ga_h[\bfx^*(s)])} \leq h^{\frac{1}{3}}, \ \mbox{for } 0 \leq s \leq  t \Big\}.\nonumber
\end{align}
In order to be able to apply the results of the previous section we observe that we can assume on $[0,\widehat{T}_h)$ 
\begin{displaymath}
\Vert e_x \Vert_{W^{1,\infty}(\Gamma_h[\bfx^*])} \leq \min\Big(\frac{1}{4}, \mu\Big), \; \Vert e_u \Vert_{L^\infty(\Gamma_h[\bfx^*])} \leq \mu, \; \Vert v_h \Vert_{W^{1,\infty}(\Gamma_h[\bfx])} \leq C,
\end{displaymath}
provided that $0 < h \leq h_1$ and $h_1 \leq h_0$ is chosen sufficiently small. Here, the constant $\mu$ appears in Lemma \ref{lemma:solution-dependent mass matrix diff}. Our aim is 
to derive an energy type estimate for $(e_x,e_v,e_u)$ which simultaneously proves that $\widehat{T}_h=T$ as well as the error bounds claimed in Theorem \ref{theorem: error estimates - modified aMCF}. To
do so we will work with the matrix-vector form for the components of the error which we obtain from  \eqref{eq:matrix-vector form - modified aMCF} and \eqref{eq:defect definition - modified aMCF}: 
\begin{subequations}
	\label{eq:error equations - modified aMCF}
	\begin{align}
	         \label{eq:error equations - modified aMCF - x}
		\dotex = &\ \ev , \\
		\label{eq:error equations - modified aMCF - v}
		\ev = &\ \big(\bfV \bullet \bfn - \bfV^* \bullet \bfn^* \big) - \dv , \\
		\label{eq:error equations - modified aMCF - u}
		\bfM(\bfx,\bfu) \doteu + \bfA(\bfx,\bfu) \eu & \\
		& \hspace{-3cm}  = 	   \big( \bfM(\xs,\us) -  \bfM(\bfx,\bfu) \big) \dotus 
		  + \big(    \bfA(\xs,\us) -  \bfA(\bfx,\bfu)  \big) \us \nonumber \\
	& \hspace{-2.6cm}	 + \big(\bffb(\bfx,\bfu) - \bffb(\xs,\us) \big) 
		 - \bfM(\xs) \du . \nonumber 
	\end{align}
\end{subequations}
Note that we have written $\bfM(\bfx,\bfu)$ for $\bfM^{[d+2]}(\bfx,\bfu)$, and analogously for all other matrices in order to simplify the notation. To begin,
let us use the argument in  \cite[(B), p. 628]{KLL_Willmore}  to show that
\begin{displaymath}
 \Vert e_v \Vert_{H^1(\Gamma_h[\bfx^*])} \leq C \Vert e_u \Vert_{H^1(\Gamma_h[\bfx^*])}  + C \Vert d_v \Vert_{H^1(\Gamma_h[\bfx^*])} 
\end{displaymath}
and hence, recalling \eqref{eq:dvuest}
\begin{equation} \label{eq:vstab}
 \Vert e_v(t) \Vert_{H^1(\Gamma_h[\bfx^*(t)])}^2 \leq C h^{2k} + C  \Vert e_u(t) \Vert_{H^1(\Gamma_h[\bfx^*(t)])}^2, \quad 0 \leq t < \widehat{T}_h.
 \end{equation}
 Next, observing that $e_x(0)=0$, $\mat_h e_x= e_v$ as well as $\Vert \nabla_{\Gamma_h} v_h^* \Vert_{L^\infty(\Gamma_h[\xs])} \leq C$ we may estimate with the help of \eqref{eq:vstab}
 \begin{eqnarray}
 \lefteqn{ 
 \Vert e_x(t) \Vert_{H^1(\Gamma_h[\xs(t)])}^2 } \label{eq:xstab} \\
 & \leq &  C \int_0^t \bigl(  \Vert e_x(s) \Vert_{H^1(\Gamma_h[\bfx^*(s)]} 
 \Vert \mat_h e_x(s) \Vert_{H^1(\Gamma_h[\bfx^*(s)]}+  \Vert e_x(s) \Vert_{H^1(\Gamma_h[\bfx^*(s)])}^2 \bigr) \d s  \nonumber \\
 & \leq & C \int_0^t \bigl(  \Vert e_x(s) \Vert_{H^1(\Gamma_h[\bfx^*(s)]}^2 + \Vert e_v(s) \Vert_{H^1(\Gamma_h[\bfx^*(s)])}^2 \bigr) \d s  \nonumber \\
 & \leq & C h^{2k} + C  \int_0^t \bigl(  \Vert e_x(s) \Vert_{H^1(\Gamma_h[\bfx^*(s)]}^2 +  \Vert e_u(s) \Vert_{H^1(\Gamma_h[\bfx^*(s)])}^2 \bigr) \d s . \nonumber
 \end{eqnarray}
Similarly, using that $\Vert e_u(0) \Vert_{L^2(\Gamma_h[\xs(0)])} \leq C h^k$ we derive
\begin{eqnarray}
 \lefteqn{ 
 \Vert e_u(t) \Vert_{L^2(\Gamma_h[\xs(t)])}^2 } \label{eq:ustab} \\
 & \leq & C h^{2k} +   C \int_0^t \bigl(  \Vert e_u(s) \Vert_{L^2(\Gamma_h[\bfx^*(s)]} 
 \Vert \mat_h e_u(s) \Vert_{L^2(\Gamma_h[\bfx^*(s)]}+  \Vert e_u(s) \Vert_{L^2(\Gamma_h[\bfx^*(s)])}^2 \bigr) \d s  \nonumber \\
 & \leq & \eps \int_0^t  \Vert \mat_h e_u(s) \Vert_{L^2(\Gamma_h[\bfx^*(s)]} ^2\d s + C h^{2k} + C_\eps \int_0^t  \Vert  e_u(s) \Vert_{L^2(\Gamma_h[\bfx^*(s)]}^2 \d s. \nonumber
 \end{eqnarray}

\noindent
In order to control $\nabla_{\Gamma_h} e_u$   we multiply  \eqref{eq:error equations - modified aMCF - u} by $\doteu$ and deduce that
\begin{eqnarray}
\lefteqn{ \doteu^T \bfM(\bfx,\bfu) \doteu + \doteu^T \bfA(\bfx,\bfu) \eu } \nonumber \\
& = 	&     \doteu^T  \big(  \bfM(\xs,\us) - \bfM(\bfx,\bfu)   \big) \dotus  
	  +  \doteu^T  \big(  \bfA(\xs,\us) - \bfA(\bfx,\bfu)  \big) \us  \nonumber \\
	&  & 	  + \doteu^T \big(\bffb(\bfx,\bfu) - \bffb(\xs,\us) \big) 
		 - \doteu^T \bfM(\xs) \du \nonumber \\
		 & = & \sum_{i=1}^4 T_i.  \label{eq:errsystem}
\end{eqnarray}
The left-hand side can be estimated with the help of \eqref{eq:Mxunormequiv},  \eqref{eq:atime1} and the fact that
$\Vert e_u \Vert_{W^{1,\infty}(\Gamma_h[\xs])} \leq 1$:
\begin{eqnarray} 
\lefteqn{ 
 \doteu^T \bfM(\bfx,\bfu) \doteu + \doteu^T \bfA(\bfx,\bfu) \eu } \nonumber \\
& = & \Vert \doteu \Vert_{\bfM(\bfx,\bfu)}^2 + \frac{1}{2} \diff \Vert \eu \Vert_{A(\bfx,\bfu)}^2 - \frac{1}{2} \eu^T \diff \bfA(\bfx,\bfu) \eu \nonumber  \\
& \geq & c \Vert \mat_h e_u \Vert_{L^2(\Gamma_h[\bfx^*])}^2 + \frac{1}{2} \diff \Vert \eu \Vert_{A(\bfx,\bfu)}^2 - 
C  \| e_u \|_{H^1(\Gamma_h[\bfx^*])}  \bigl( \| e_u  \|_{H^1(\Gamma_h[\bfx^*])} + \Vert \mat_h e_u  \Vert_{L^2(\Gamma_h[\bfx^*])}  \bigr) \nonumber \\
& \geq & \frac{c}{2}  \Vert \mat_h e_u \Vert_{L^2(\Gamma_h[\bfx^*])}^2 + \frac{1}{2} \diff \Vert \eu \Vert_{A(\bfx,\bfu)}^2 - C \, \| e_u \|_{H^1(\Gamma_h[\bfx^*])}^2 . \label{eq:lhs}
\end{eqnarray}
Next, \eqref{eq:mdif1} and \eqref{eq:mdif2} together with the fact that $\Vert \mat_h u_h^* \Vert_{W^{1,\infty}(\Gamma_h[\xs])} \leq C$ imply that
\begin{eqnarray}
 T_1  & \leq & C \bigl( \Vert e_u \Vert_{L^2(\Gamma_h[\bfx^*])} + \Vert e_x \Vert_{H^1(\Gamma_h[\bfx^*])} \bigr) \Vert  \mat_h e_u \Vert_{L^2(\Gamma_h[\bfx^*])} \nonumber \\
& \leq & \eps  \Vert  \mat_h e_u \Vert_{L^2(\Gamma_h[\bfx^*])}^2 + C_\eps \bigl( \Vert e_u \Vert_{L^2(\Gamma_h[\bfx^*])}^2 + \Vert e_x \Vert_{H^1(\Gamma_h[\bfx^*])}^2 \bigr) . \label{eq:t1}
\end{eqnarray}
Observing that  $\Vert \mat_h u_h^* \Vert_{W^{1,\infty}(\Gamma_h[\xs])} \leq C$ we infer from \eqref{eq:atime2},  \eqref{eq:atime3}, \eqref{eq:adif1}  and \eqref{eq:adif2}
\begin{eqnarray}
T_2 & = & \diff \bigl( \eu^T \big( \bfA(\xs,\us) - \bfA(\bfx,\bfu)  \big) \us \bigr) - \eu^T \diff \big(  \bfA(\xs,\us) - \bfA(\bfx,\bfu)   \big) \us \nonumber \\
& & \quad  - \,  \eu^T \big(  \bfA(\xs,\us) - \bfA(\bfx,\bfu)   \big) \dotus \nonumber \\
& \leq & \diff \bigl( \eu^T \big(  \bfA(\xs,\us)  - \bfA(\bfx,\bfu)  \big) \us \bigr) \nonumber   \\
& & \quad +
C \bigl(  \| e_u \|_{L^2(\Gamma_h[\bfx^*])} + \Vert \mat_h e_u  \Vert_{L^2(\Gamma_h[\bfx^*])}  \bigr) \| e_u \|_{H^1(\Gamma_h[\bfx^*])}  \nonumber \\
& & \quad + C \bigl(   \| e_x \|_{H^1(\Gamma_h[\bfx^*])}  + \Vert e_v \Vert_{H^1(\Gamma_h[\bfx^*])}  \bigr)  \| e_u \|_{H^1(\Gamma_h[\xs])} \nonumber \\
& \leq & \diff \bigl( \eu^T \big(  \bfA(\xs,\us)   - \bfA(\bfx,\bfu) \big) \us \bigr) + \eps  \Vert  \mat_h e_u \Vert_{L^2(\Gamma_h[\bfx^*])}^2 \nonumber \\
& & \quad + C_\eps \bigl( h^{2k} +  \Vert e_u \Vert_{H^1(\Gamma_h[\bfx^*])}^2+  \Vert e_x \Vert_{H^1(\Gamma_h[\bfx^*])}^2 \bigr), \label{eq:t2}
\end{eqnarray}
where we used \eqref{eq:vstab} in the last step. Arguing as in \cite[(A) (vii)]{MCF} one shows that
\begin{eqnarray}  
T_3 & \leq & C \Vert \mat_h e_u \Vert_{L^2(\Gamma_h[\bfx^*])} \bigl( \Vert e_u \Vert_{H^1(\Gamma_h[\bfx^*])} + \Vert e_x \Vert_{H^1(\Gamma_h[\bfx^*])} \bigr) \nonumber \\
& \leq &  \eps  \Vert  \mat_h e_u \Vert_{L^2(\Gamma_h[\bfx^*])}^2 + C_\eps \bigl( \Vert e_u \Vert_{H^1(\Gamma_h[\bfx^*])}^2 + \Vert e_x \Vert_{H^1(\Gamma_h[\bfx^*])}^2 \bigr) . \label{eq:t6}
\end{eqnarray}
Finally, \eqref{eq:dvuest} implies that
\begin{equation} \label{eq:t7}
T_4 \leq \Vert  \mat_h e_u \Vert_{L^2(\Gamma_h[\bfx^*])} \Vert d_u   \Vert_{L^2(\Gamma_h[\bfx^*])} \leq \eps  \Vert  \mat_h e_u \Vert_{L^2(\Gamma_h[\bfx^*])}^2 + C_\eps h^{2k}.
\end{equation}
If we insert \eqref{eq:t1}--\eqref{eq:t7} into \eqref{eq:lhs},  choose $\eps>0$ small enough we obtain after integration with respect to time and
recalling that $\Vert e_u(0) \Vert_{H^1(\Gamma_h[\xs(0)])} \leq C h^k$ as well as \eqref{eq:adif1}, \eqref{eq:adif2}
\begin{eqnarray*}
\lefteqn{ \frac{c}{4} \int_0^t \Vert  \mat_h e_u \Vert_{L^2(\Gamma_h[\bfx^*])}^2 \d s + \frac{1}{2}  \Vert \eu(t) \Vert_{A(\bfx(t),\bfu(t))}^2 } \\
& \leq & C h^{2k} +  \bigl( \eu(t)^T \big(  \bfA(\xs(t),\us(t))   - \bfA(\bfx(t),\bfu(t)) \big) \us(t) \bigr) \\
& & \; + C \int_0^t \bigl( \Vert e_u \Vert_{H^1(\Gamma_h[\bfx^*])}^2+  \Vert e_x \Vert_{H^1(\Gamma_h[\bfx^*])}^2 \bigr)\d s \\
& \leq &C \Vert e_u(t) \Vert_{H^1(\Gamma_h[\xs(t)])} \bigl( \Vert e_u(t) \Vert_{L^2(\Gamma_h[\xs(t)])} +\Vert e_x(t) \Vert_{H^1(\Gamma_h[\xs(t)])} \bigr) \\
& & \; + C h^{2k} + C \int_0^t \bigl( \Vert e_u \Vert_{H^1(\Gamma_h[\bfx^*])}^2+  \Vert e_x \Vert_{H^1(\Gamma_h[\bfx^*])}^2 \bigr)\d s \\
& \leq & \eps \Vert \nabla_{\Gamma_h} e_u(t) \Vert_{L^2(\Gamma_h[\xs(t)])}^2 + C_\eps \bigl( \Vert e_u(t) \Vert_{L^2(\Gamma_h[\xs(t)])}^2 +\Vert e_x(t) \Vert_{H^1(\Gamma_h[\xs(t)])}^2 \bigr) \\
& & \; + C h^{2k} + C \int_0^t \bigl( \Vert e_u \Vert_{H^1(\Gamma_h[\bfx^*])}^2+  \Vert e_x \Vert_{H^1(\Gamma_h[\bfx^*])}^2 \bigr)\d s.
\end{eqnarray*}
If we now use \eqref{eq:Axunormequiv} and choose $\eps>0$ small enough we infer that
\begin{eqnarray}
\lefteqn{ \hspace{-1cm} 
\int_0^t \Vert  \mat_h e_u \Vert_{L^2(\Gamma_h[\bfx^*])}^2 \d s + \Vert \nabla_{\Gamma_h} e_u(t) \Vert_{L^2(\Gamma_h[\xs(t)])}^2 } \label{eq:ustab1} \\
& \leq & \hat c  \bigl( \Vert e_u(t) \Vert_{L^2(\Gamma_h[\xs(t)])}^2 +\Vert e_x(t) \Vert_{H^1(\Gamma_h[\xs(t)])}^2 \bigr) \nonumber \\
&  & \; + C h^{2k} + C \int_0^t \bigl( \Vert e_u \Vert_{H^1(\Gamma_h[\bfx^*])}^2+  \Vert e_x \Vert_{H^1(\Gamma_h[\bfx^*])}^2 \bigr)\d s. \nonumber 
\end{eqnarray} 
Multiplying the inequalities \eqref{eq:xstab}, \eqref{eq:ustab} by $\hat c+1$ adding the result to \eqref{eq:ustab1} we finally obtain after choosing $\eps$ sufficiently small
\begin{displaymath}
\Vert (e_x,e_u)(t) \Vert_{H^1(\Gamma_h[\xs(t)])}^2 \leq C h^{2k} +  C \int_0^t  \Vert (e_x,e_u)(s) \Vert_{H^1(\Gamma_h[\bfx^*(s)])}^2 \d s, \; 0 \leq t <  \widehat{T}_h,
\end{displaymath}
from which we deduce with the help of Gronwall's inequality and \eqref{eq:vstab} that 
\begin{equation} \label{eq:err1}
\Vert (e_x,e_v,e_u)(t) \Vert_{H^1(\Gamma_h[\xs(t)])} \leq C h^{k}, \quad 0 \leq t < \widehat{T}_h.
\end{equation}
Since $k \geq 2$ and $d \in \lbrace 1,2,3 \rbrace$ we have that  $k - \frac{d}{2} \geq \tfrac{1}{2}$ so that an inverse estimate implies 
\begin{displaymath}
\Vert (e_x,e_v,e_u)(t) \Vert_{W^{1,\infty}(\Gamma_h[\xs(t)])} \leq C h^{k-\frac{d}{2}} \leq C h^{\frac{1}{2}} \leq \frac{1}{2} h^{\frac{1}{3}},  \quad 0 \leq t < \widehat{T}_h ,
\end{displaymath}
for $0<h \leq h_2$ if  $h_2 \leq h_1$ is small enough. Recalling \eqref{eq:hatTh} we therefore must have $\widehat{T}_h=T$. In order to prove
the error estimates in Theorem  \ref{theorem: error estimates - modified aMCF} we use 
\begin{eqnarray*}
\lefteqn{ \hspace{-1cm}
\Vert (\nu_h^L,V_h^L)- (\nu,V) \Vert_{H^1(\Gamma[X])} = \Vert u_h^L - u \Vert_{H^1(\Gamma[X])} } \\
& \leq & \Vert u_h^L - (u_h^*)^\ell \Vert_{H^1(\Gamma[X])} + \Vert (u_h^*)^\ell - u \Vert_{H^1(\Gamma[X])} \\
& \leq & C \Vert e_u \Vert_{H^1(\Gamma_h[\xs])} + Ch^k \leq C h^k ,
\end{eqnarray*} 
where we applied \eqref{eq:ritzest}, \eqref{eq:err1} and the definition of $u_h^L$. \qed

\section{A stabilized algorithm}
\label{section:regularization}

%

For ``crystalline-like'' anisotropies, e.g.~those in Figure~\ref{fig:BGNwulff} (with small $\eps$), the proposed algorithm greatly benefits from a stabilization of the second derivative of the anisotropy function.

In the weak formulation \eqref{eq:coupled system - weak - modified aMCF}, let us
\begin{equation}
\label{eq:regulartization}
	\text{replace} \quad \gamma''(\nu) \qquad \text{with} \qquad \widehat{\ga''}(\nu) := \ga''(\nu) + \nu \nu^T .
\end{equation}

It is crucial to note that $\nu \nu^T$ projects onto the normal space, hence it does not change the $H^1$ bilinear form and the 
original continuous equations \eqref{eq:evolution eq - aMCF} and \eqref{eq:coupled system - weak - modified aMCF} do not change.  Naturally, since the discrete normal $\nu_h$ and the true normal of the discrete surface $\nu_{\Ga_h}$ differ, this is not any more true for the semi-discretization:
\begin{equation*}
	\widehat{\ga''}(\nu_h) = \ga''(\nu_h) + \nu_h \nu_h^T .
\end{equation*}
This slight difference makes the corresponding stiffness-matrix 
\begin{equation*}
	\widehat{\bfA}(\bfx,\bfu) \quad \text{more coercive}.
\end{equation*}

For \emph{isotropic} mean curvature flow, $\gamma(w) = |w|$, from Remark~\ref{remark:isotropic gamma} (cf.~Example~\ref{example:anisotropies}) we recall 
\begin{equation*}
	\ga''(\nu) = I_{d+1} - \nu \nu^T .
\end{equation*}
Using this in the discretization \eqref{eq:semidiscretization - modified aMCF}, the normal projection would remain. On the other hand, using the regularization would set $\widehat{\ga''}(\nu) = \ga''(\nu) + \nu \nu^T = I_{d+1}$. 
The mean curvature flow algorithm proposed in \cite{MCF} does not use  $\ga''(\nu_h) = I_{d+1} - \nu_h \nu_h^T$, but the identity matrix which coincides with $\widehat{\ga''}(\nu_h) = I_{d+1}$.

It is crucial to observe that the results of Lemma~\ref{lemma:solution-dependent mass matrix diff}--\ref{lemma:quasi-linear stiffness matrix diff} remain valid for the stabilised stiffness-matrix $\widehat{\bfA}(\bfx,\bfu)$, since the used coercivity and Lipschitz properties of $\ga''$ are, respectively, improved and unaffected (up to constants) in the case of $\widehat{\ga''}$.

Therefore, we are confident that the statement of Theorem~\ref{theorem: error estimates - modified aMCF} holds for the stabilised algorithm as well.

\section{Numerical experiments}
\label{sec:numericalexperiments}

We performed  numerical simulations and experiments for the anisotropic mean curvature flow \eqref{eq:anisotropicMCF - problem} formulated as the coupled system \eqref{eq:coupled system - weak - modified aMCF} in which: 
\begin{itemize}
	\item[-] The rate of convergence in an example involving a self-similar exact solution is studied in order to illustrate the  theoretical results of Theorem~\ref{theorem: error estimates - modified aMCF}.
	\item[-] Various anisotropic mean curvature flows with ``crystalline-like'', and asymmetric anisotropies, reporting on surface evolutions and on the decreasing energy $\calE_{\ga}(\Ga) = \int_\Ga \ga(\nu)$, see \eqref{eq:anenergy}.
	\item[-] We report on anisotropic mean curvature flows with asymmetric anisotropy, using $\ga$ from equation~(8.9) in \cite{DeckelnickDE2005} with $p = 4$ instead of $2$.
	\item[-] We report on the effects of the regularization from Section~\ref{section:regularization}.
\end{itemize}

The numerical experiments use quadratic evolving surface finite elements, implemented in the Matlab package  $\ell$FEM \cite{ellFEM}. For the computation of all finite element matrices and vectors it uses high-order quadratures so that the resulting quadrature error does not feature in the discussion of the accuracies of the schemes. 
The temporal discretization uses linearly-implicit BDF methods of order $1,\dotsc,5$, see Section~\ref{section:BDF}. 
The initial meshes were all generated using DistMesh \cite{distmesh}, without taking advantage of any symmetry of the surface.

\subsection{Linearly-implicit backward differentiation formulae}
\label{section:BDF}

For the time discretization of the system of ordinary differential equations \eqref{eq:matrix-vector form - modified aMCF} we use a $q$-step linearly implicit backward difference formula (BDF method). For a step size $\tau>0$, and with $t_n = n \tau \leq T$, we determine the approximations to all variables $\bfx^n$ to $\bfx(t_n)$, $\bfv^n$ to $\bfv(t_n)$,  and $\bfu^n$ to $\bfu(t_n)$, by the fully discrete system of \emph{linear} equations, for $n \geq q$,
\begin{subequations}
	\label{eq:BDF}
	\begin{align}
		\label{eq:BDF -- x}
		\dot \bfx^n =&\ \bfv^n , \\
		\label{eq:BDF -- v}
		\bfv^n =&\ \bfV^n \bullet \bfn^n , \\
		\label{eq:BDF -- u}
		\bfM^{[d+2]}(\widetilde \bfx^n,\widetilde \bfu^n) \dot{\bfu}^n + \bfA^{[d+2]}(\widetilde \bfx^n,\widetilde \bfu^n) \bfu^n =&\ \bff(\widetilde \bfx^n,\widetilde \bfu^n) , 
	\end{align}
\end{subequations}
where we denote the discretised time derivatives and the extrapolated values by
\begin{equation}
	\label{eq:backward differences def}
	\dot \bfx^n = \frac{1}{\tau} \sum_{j=0}^q \delta_j \bfx^{n-j} , \andquad \widetilde \bfx^n = \sum_{j=0}^{q-1} \gamma_j \bfx^{n-1-j} , \qquad n \geq q .
\end{equation}
Both notations may also be used for all other variables.

The starting values $\bfx^i$ and $\bfu^i$ ($i=0,\dotsc,q-1$) are assumed to be given. 
The initial values can be precomputed using either a lower order method with smaller step sizes, or an implicit Runge--Kutta method.

The method is determined by its coefficients, given by 
\begin{equation*}
	\delta(\zeta)=\sum_{j=0}^q \delta_j \zeta^j=\sum_{\ell=1}^q \frac{1}{\ell}(1-\zeta)^\ell 
	\andquad 
	\gamma(\zeta) = \sum_{j=0}^{q-1} \gamma_j \zeta^j = (1 - (1-\zeta)^q)/\zeta .
\end{equation*}
The classical BDF method is known to be zero-stable for $q\leq6$ and to have order $q$; see \cite[Chapter~V]{HairerWannerII}.
This order is retained, for $q \leq 5$ see \cite{LubichMansourVenkataraman_bdsurf}, also by the linearly-implicit variant using the above coefficients $\gamma_j$; cf.~\cite{AkrivisLubich_quasilinBDF,AkrivisLiLubich_quasilinBDF}.
In \cite{Akrivisetal_BDF6}, the multiplier techniques of \cite{NevanlinnaOdeh} have been recently extended, via a new approach, to the six-step BDF method.

\subsection{Convergence test using a self-similar ellipsoid solution}

Following \cite[Section~5.1]{BGN_anisotropic} we perform a numerical experiment for a self-similar ellipsoid. 
For  $\eps > 0$ we define
\begin{equation}
\label{eq:num exp - gamma and beta}
	\ga(w) = \sqrt{w \cdot G w}, \text{ with } G = \diag(1,\eps^2,\eps^2) , \andquad  \beta(w) = \frac{1}{\ga(w)} ,
\end{equation}
that is the anisotropy $\ga$ from Example~\ref{example:anisotropies} (b) with a special diagonal matrix $G$.
 With this choice, one obtains 
 $$\ga^*(w) = \sqrt{w \cdot G^* w}, \quad \text{ with } \quad G^* = \diag\Big(1,\frac1{\eps^2},\frac1{\eps^2}\Big) $$
 and as discussed in the introduction the boundary of the one ball of $\ga^*$ shrinks in a self-similar way.
 We hence choose the set of all points with $\ga^*(w)=1$ as initial surface $\Ga^0$, i.e., we set
  $\Ga^0 := \{ x \in \R^3 \mid d(x,0) = 0 \}$ expressed with the help of  the level-set function
\begin{equation*}
	d(x,0) = x_1^2 + x_2^2 / \eps^2 + x_3^2 / \eps^2 - 1^2 .
\end{equation*}
The exact solution of 
$$\frac 1{\gamma(\nu)} V =- H_\gamma$$
is the hypersurface $\Ga(t)$ which is for $t\in[0,0.25)$  the zero level-set of the level-set function
\begin{equation*}
	d(x,t) = x_1^2 + x_2^2 / \eps^2 + x_3^2 / \eps^2 - ( \sqrt{1-4t} )^2 .
\end{equation*}
The fact that $\Ga(t)$ solves $\frac 1{\gamma(\nu)} V =- H_\gamma$ can be shown as follows. Compute the outward normal field $\nu$ and normal velocity $V$ with the help of the formulas, see, e.g.,~\cite[Section~2.1]{DziukElliott_ESFEM} (using that $\nb d \neq 0$),
\begin{equation}
\label{eq:V and nu def from level set function}
	\nu = \frac{\nb d}{|\nb d|}, \andquad V = - \frac{\pa_t d}{|\nb d|} .
\end{equation}
A straightforward computation, using $\nabla_\Gamma \cdot \gamma'(\nu) =\nabla \cdot  \gamma'(\nu) -
\nu\cdot ((\nabla \gamma'(\nu) )^T \nu )
$  gives $H_\gamma= \frac 2{\sqrt{1-4t}}= - \frac 1{\gamma(\nu)} V  $
which shows that $\Ga(t)$ on $[0,0.25) $ solves $\frac 1{\gamma(\nu)} V =- H_\gamma$.
Note that the surfaces  $\Gamma(t)$ are identical to the sets $\sqrt{1-4t} \Ga^0$ and satisfy $d(\cdot,t) = 0$ for all $t \in [0,0.25]$.

We report on a convergence experiment for the above described self-similar anisotropic surface evolution with $\eps = 0.5$. 
We started the algorithms from the nodal interpolations of the exact initial values $\Ga^0$, $\nu^0$, and $V^0 = - \beta(\nu^0)^{-1} \nbg \cdot ( \ga'(\nu^0)) \nu^0$, subsequent initial values $i=1,\dotsc,q-1$ were obtained by lower order BDF steps. 
In order to illustrate the convergence results of Theorem~\ref{theorem: error estimates - modified aMCF}, we have computed the errors between the numerical solution \eqref{eq:BDF} and the exact solution of \eqref{eq:anisotropicMCF - problem}. The exact nodal values of $\Ga[X]$ were obtained by solving the ODE $\pa_t X = V \nu$, with \eqref{eq:V and nu def from level set function}, using, since the surface evolution is not a stiff problem, via the explicit Adams method of the same order.

The numerical solutions and energy obtained by our algorithm are plotted in Figure~\ref{fig:solutions_conv}.

\begin{figure}[htbp]
	\centering
	\includegraphics[width=\textwidth, trim={85 30 55 0}, clip]{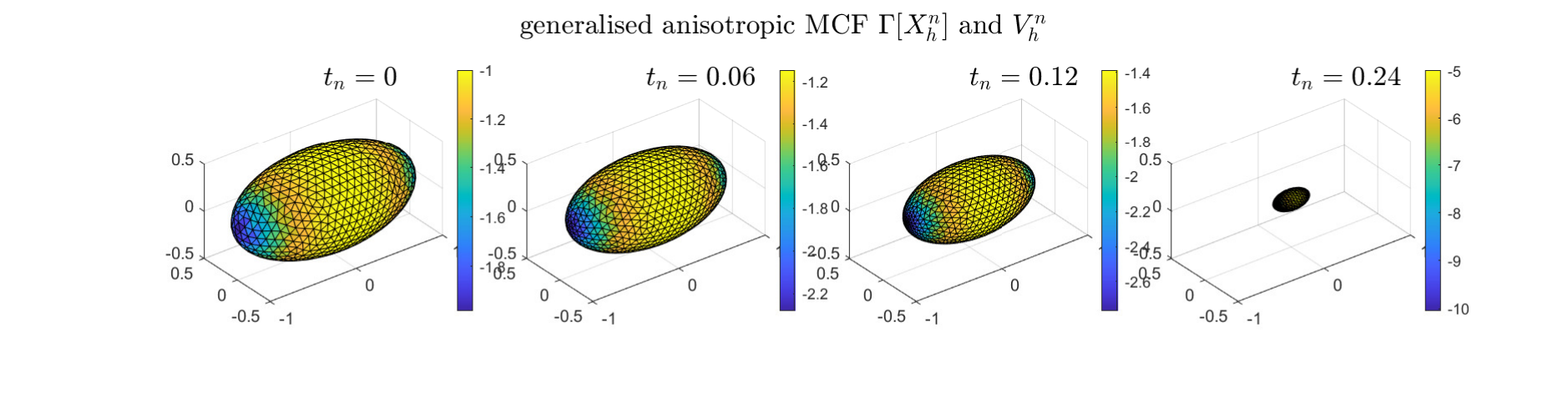}
	
	\vspace{2mm} 
	
	\includegraphics[width=\textwidth, trim={85 5 55 30}, clip]{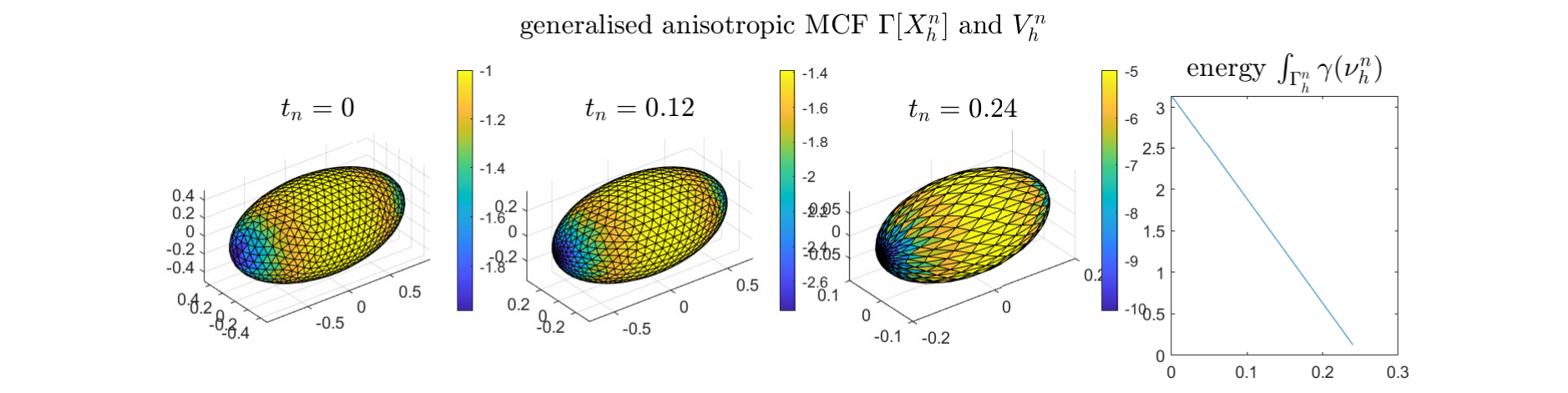}
	
	\caption{Numerical solutions of the self-similar flow using the BDF2 / quadratic ESFEM discretization (without and with rescaling), and the decay of the anisotropic energy.}
	\label{fig:solutions_conv}
\end{figure}

In Figure~\ref{fig:anisotropic flow - convplot space - ellipsoid - self-similar} and \ref{fig:anisotropic flow - convplot time - ellipsoid - self-similar} we report the errors between the numerical solution and the interpolation of the exact solution until the final time $T=0.24$, for a sequence of meshes (see plots) and for a sequence of time steps $\tau_{k+1} = \tau_k / 2$. 
The double-logarithmic plots report on the $L^\infty(H^1)$ norm of the errors against the mesh width $h$ in Figure~\ref{fig:anisotropic flow - convplot space - ellipsoid - self-similar}, and against the time step size $\tau$ in Figure~\ref{fig:anisotropic flow - convplot time - ellipsoid - self-similar}.
The lines marked with different symbols and different colours correspond to different time step sizes and to different mesh refinements in Figure~\ref{fig:anisotropic flow - convplot space - ellipsoid - self-similar} and \ref{fig:anisotropic flow - convplot time - ellipsoid - self-similar}, respectively.

In Figure~\ref{fig:anisotropic flow - convplot space - ellipsoid - self-similar} we can observe two regions: a region where the spatial discretization error dominates, matching the $\calO(h^2)$ order of convergence of Theorem~\ref{theorem: error estimates - modified aMCF} (see the reference lines), and a region, with small mesh size, where the temporal discretization error dominates (the error curves flatten out). For Figure~\ref{fig:anisotropic flow - convplot time - ellipsoid - self-similar}, the same description applies, but with reversed roles. Convergence of fully discrete method is not shown, but $\calO(\tau^q)$ is expected for the BDF methods of order $q = 1,\ldots,5$, cf.~\cite{MCF}.

The spatial convergence as shown by Figure~\ref{fig:anisotropic flow - convplot space - ellipsoid - self-similar} is in agreement with the theoretical convergence results (note the reference lines). Similarly, the temporal convergence as shown by Figure~\ref{fig:anisotropic flow - convplot time - ellipsoid - self-similar} is in agreement with the \emph{expected} convergence rate of the BDF2 method, cf.~\cite{MCF}.

\begin{figure}[htbp]
	\centering
	\includegraphics[width=\textwidth, trim={0 0 0 20}, clip]{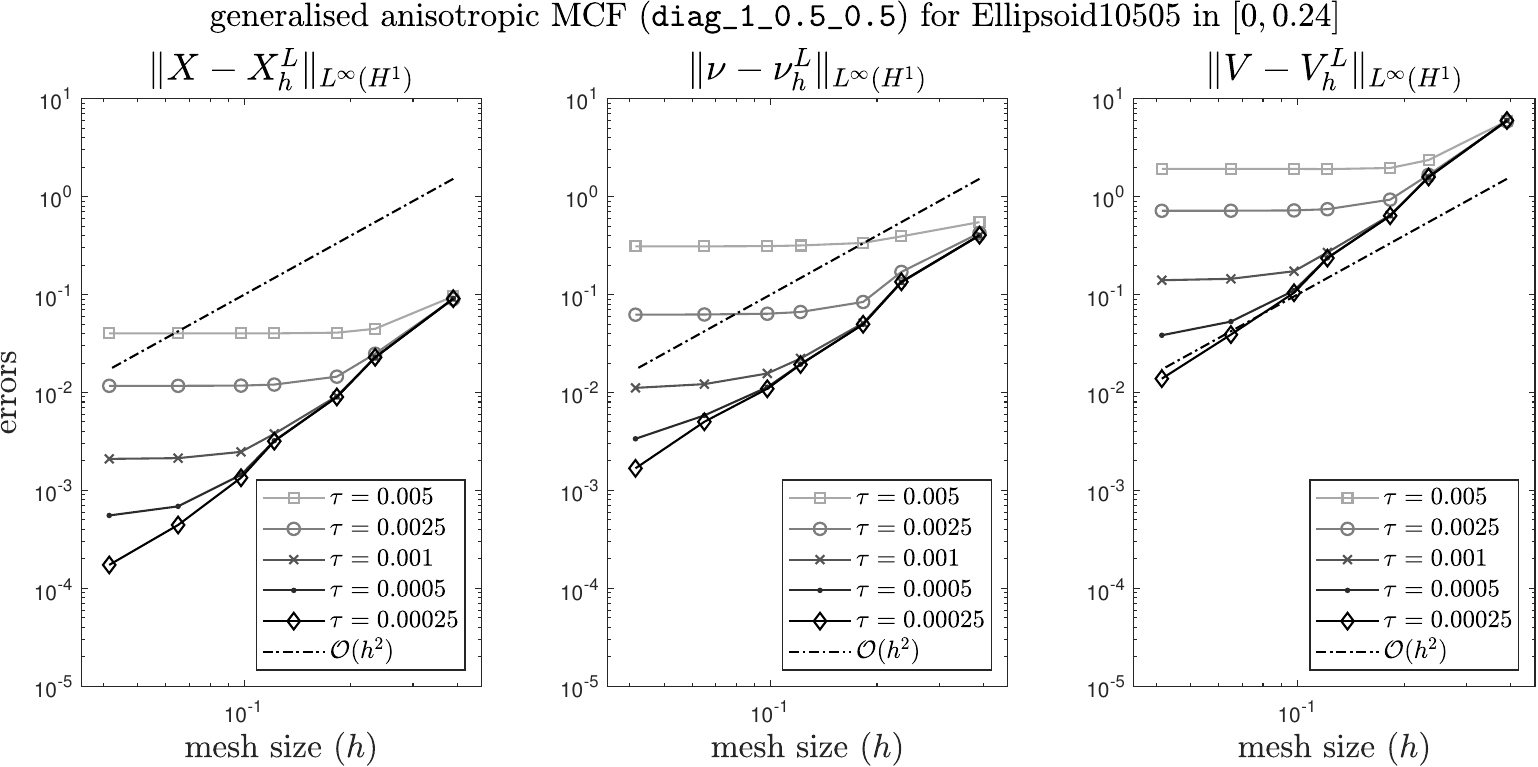}
	\caption{Spatial convergence of the BDF2 / quadratic ESFEM discretization of the anisotropic flow for the self-similar ellipsoid with $T = 0.24$.}
	\label{fig:anisotropic flow - convplot space - ellipsoid - self-similar}
\end{figure}

\begin{figure}[htbp]
	\centering
	\includegraphics[width=\textwidth, trim={0 0 0 20}, clip]{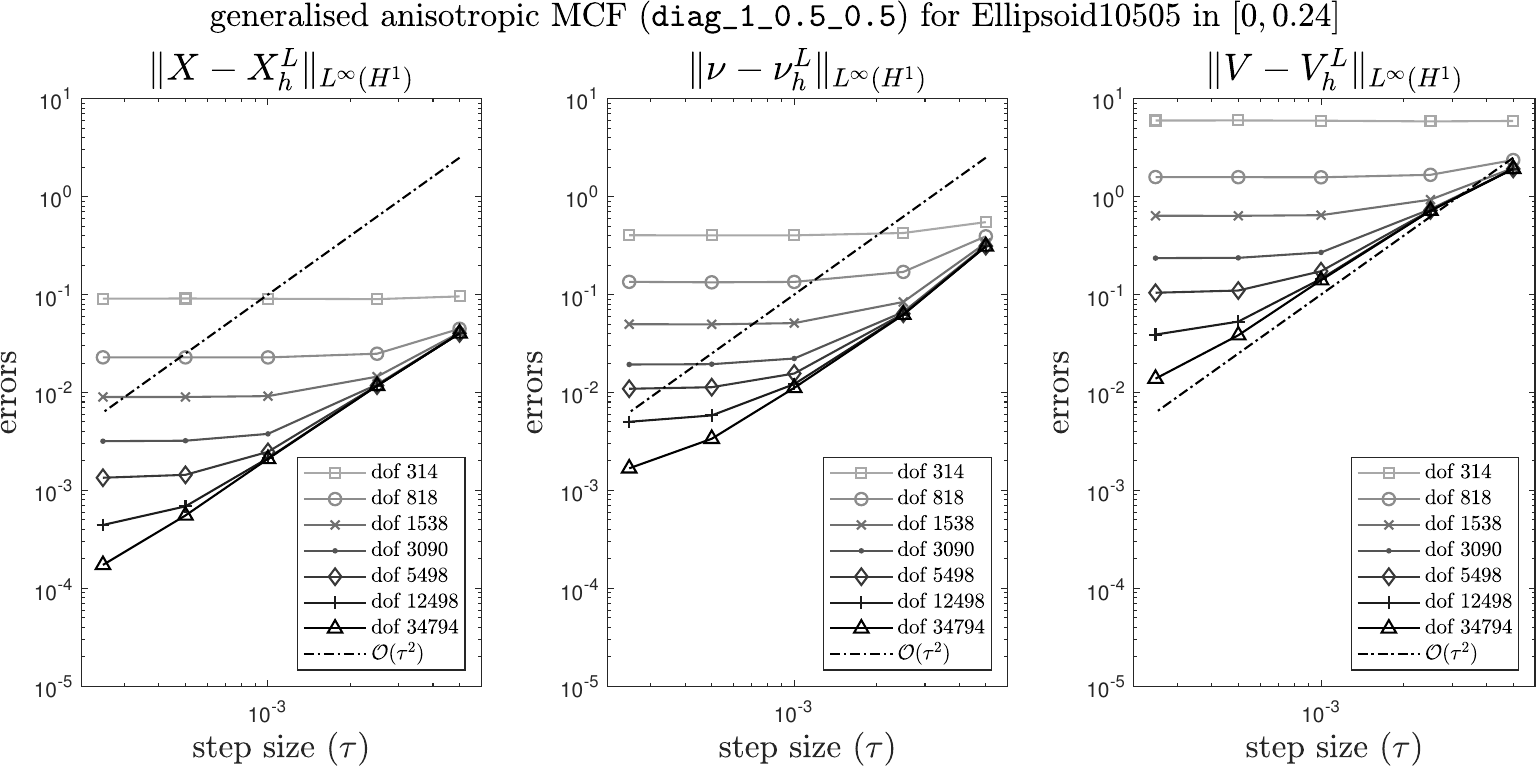}
	\caption{Temporal convergence of the BDF2 / quadratic ESFEM discretization of the anisotropic flow for the self-similar ellipsoid with $T = 0.24$.}
	\label{fig:anisotropic flow - convplot time - ellipsoid - self-similar}
\end{figure}

Most closed surfaces shrink to a Wulff-shaped point in finite time under anisotropic mean curvature flow \eqref{eq:anisotropicMCF - problem}, similarly as (uniformly convex) closed surfaces shrink to a round point in finite time under mean curvature flow, \cite[Theorem~1.1]{Huisken1984}.

Figure~\ref{fig:convergence to ellipse} reports on the evolution of a spherical initial surface under anisotropic mean curvature flow setting \eqref{eq:num exp - gamma and beta} with $\eps = 0.1$.

\begin{figure}[htbp]
	\centering
	\includegraphics[width=\textwidth, trim={85 10 45 5}, clip]{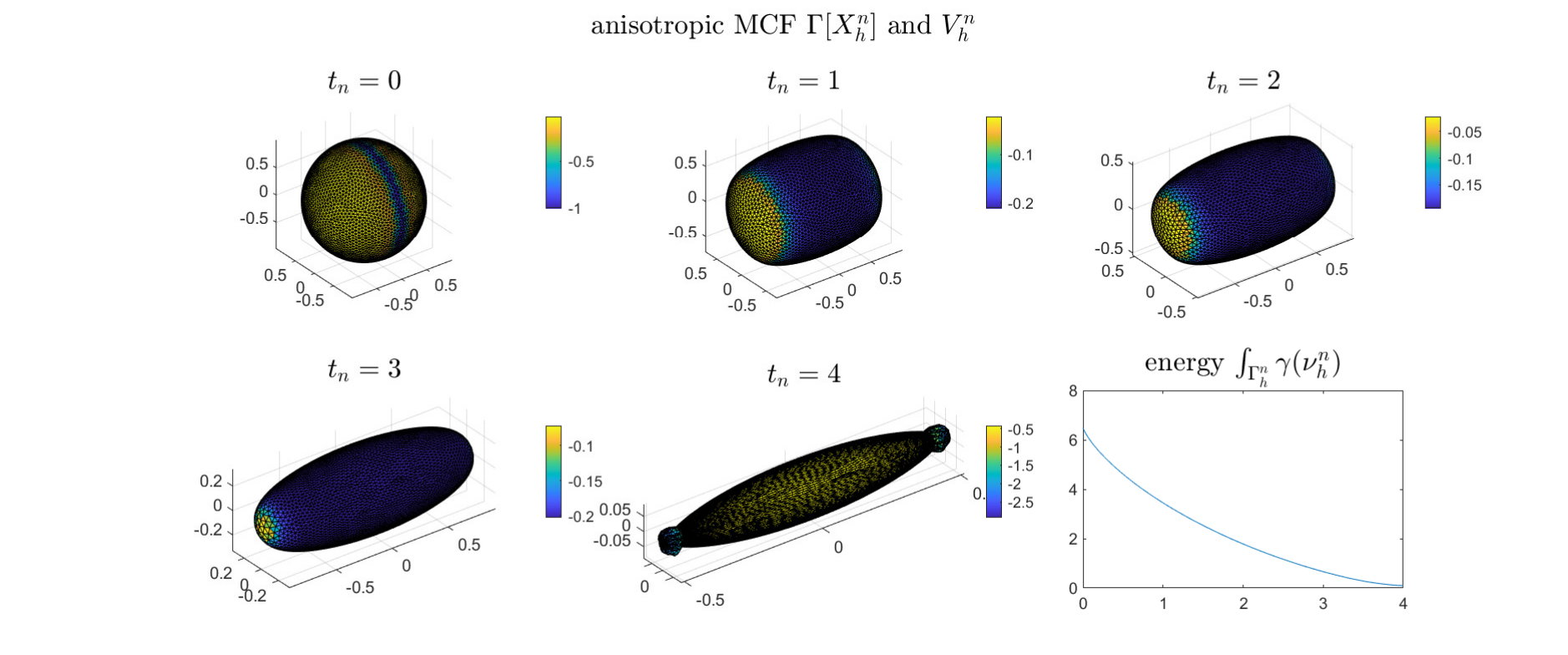}
	
	\caption{Numerical solutions of the ellipsoidal anisotropic MCF using the BDF2 / quadratic ESFEM discretization ($\eps = 0.1$, $\textnormal{dof} = 3882$, and $\tau = 10^{-3}$), and the decay of the anisotropic energy.}
	\label{fig:convergence to ellipse}
\end{figure}

\subsection{Crystalline-like anisotropies}

We have performed some numerical experiments with two ``\emph{crystalline-like}'' anisotropies from \cite{BGN_anisotropic}. In particular we use the cubic and hexagonal anisotropy function (regularized with $\eps = 0.1$) described in Example~\ref{example:anisotropies} (c), and see Figure~\ref{fig:BGNwulff} depicting the corresponding Wulff shapes and \cite[page~9]{BGN_anisotropic} for precise definitions.

Note that both these anisotropies are quite irregular (the Wulff shapes having $\eps$-regularized edges), whereas our numerical algorithm is particularly suitable for regular anisotropies, see Figure~\ref{fig:convergence to ellipse} with $\ga$ from \eqref{eq:num exp - gamma and beta} with $\eps = 0.1$ and $\beta \equiv 1$. Nevertheless, the algorithm performs quite well, see Figure~\ref{fig:cubic} and \ref{fig:hexagonal}.

\begin{figure}[htbp]
	\centering
	\includegraphics[width=\textwidth, trim={85 5 45 5}, clip]{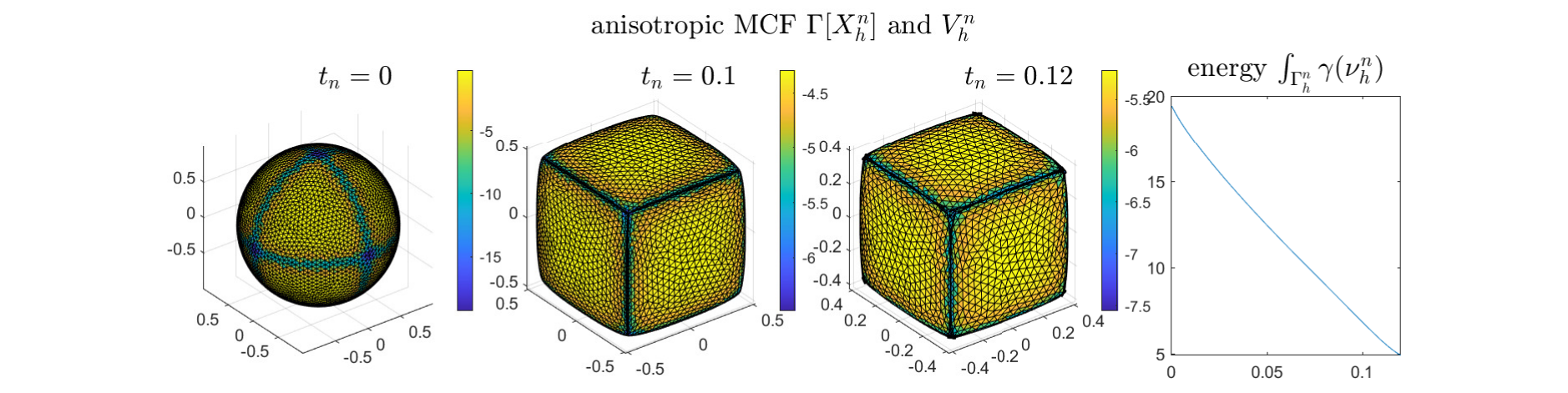}
	
	\caption{Numerical solutions of the cubic ($\eps = 0.1$) anisotropic MCF using the BDF2 / quadratic ESFEM discretization ($\textnormal{dof} = 3882$ and $\tau = 10^{-3}$), and the decay of the anisotropic energy.}
	\label{fig:cubic}
\end{figure}

\begin{figure}[htbp]
	\centering
	\includegraphics[width=\textwidth, trim={85 5 45 5}, clip]{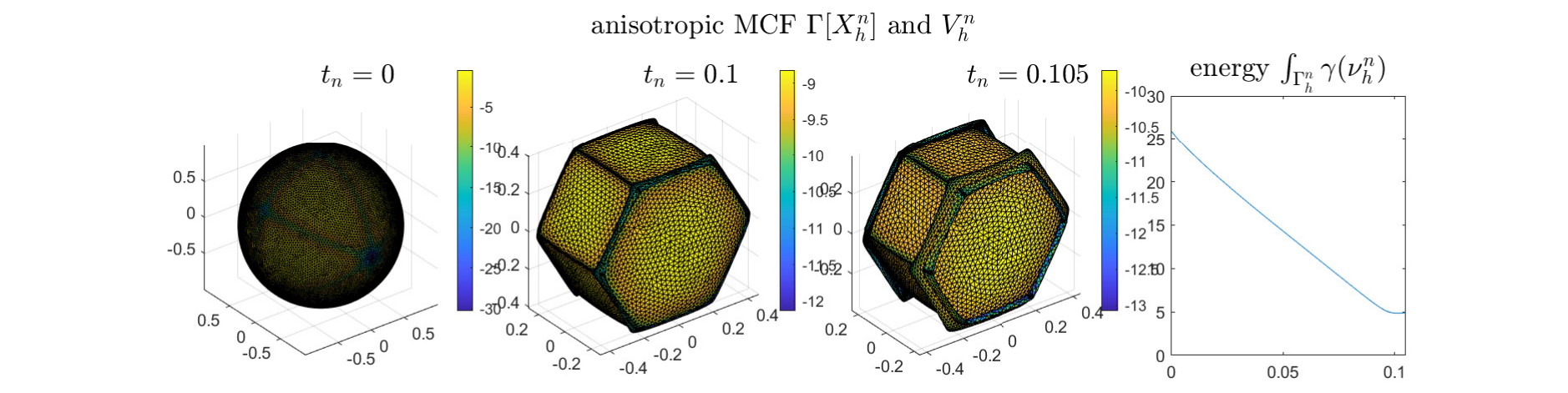}
	
	\caption{Numerical solutions of the hexagonal ($\eps = 0.1$) anisotropic MCF using the BDF2 / quadratic ESFEM discretization ($\textnormal{dof} = 12066$ and $\tau = 10^{-3}$), and the decay of the anisotropic energy.}
	\label{fig:hexagonal}
\end{figure}

\subsection{Asymmetric anisotropy}

We have performed numerical experiments for anisotropic mean curvature flow \eqref{eq:anisotropicMCF - problem} with \emph{asymmetric} anisotropy
\begin{equation}
\label{eq:asymmetric anisotropy}
	\gamma(w) = \Big( \big(5.5 + 4.5\, \mathrm{sign}(w_1) \big) w_1^p + w_2^p + w_3^p \Big)^{1/p} , \qquad \text{with} \quad p := 4,
\end{equation}
from equation~(8.9) in \cite{DeckelnickDE2005} with exponent $p = 4$ instead of $2$. See \cite[Figure~8.2]{DeckelnickDE2005} for Frank diagram and Wulff shape with $p=2$. 

In the spirit of Example~\ref{example:anisotropies} (b): Setting $a(w_1) = 5.5 + 4.5\, \mathrm{sign}(w_1)$ and
\begin{equation*}
	G(w_1) := \diag(a(w_1),1,1) \in \R^{3 \times 3} ,
\end{equation*}
and therefore rewriting
\begin{equation*}
	\gamma(w) = \Big( w \cdot G(w_1) w^{p-1} \Big)^{1/p} ,
\end{equation*}
where $w^{q}|_j = w_j^{q}$ for all $j = 1,2,3$ and any $q$.

We further keep $p \geq 4$ general, and note that (for such $p$) $(\mathrm{sign}(x)x^p)' = (|x|x^{p-1})' = p |x| x^{p-2} = p \, \mathrm{sign}(x)x^{p-1}$ and $(|x|x^{p-1})'' = p(p-1) |x| x^{p-3} = p(p-1) \, \mathrm{sign}(x)x^{p-2}$.

The derivative of $\gamma$ is then given by
\begin{equation*}
	\gamma'(w) =
	\big( \gamma(w) \big)^{1-p} 
	G(w_1) w^{p-1},
\end{equation*}
while second derivative is given by
\begin{align*}
	\gamma''(w) 
	= &\ (p-1) \big( \gamma(w) \big)^{1-p} \diag\big( G(w_1) w^{p-2} \big) \\
	&\ + (1-p) \big( \gamma(w) \big)^{1-2p} \big( G(w_1) w^{p-1} \big) \otimes \big( G(w_1) w^{p-1} \big) ,
\end{align*}
cf.~the derivatives in Example~\ref{example:anisotropies} (b).


Using the anisotropy \eqref{eq:asymmetric anisotropy} (and $\beta \equiv 1$), the surface evolution and the decay of the anisotropic energy are reported in Figure~\ref{fig:asymmetric}.
\begin{figure}[htbp]
	\centering
	\includegraphics[width=\textwidth, trim={85 5 45 5}, clip]{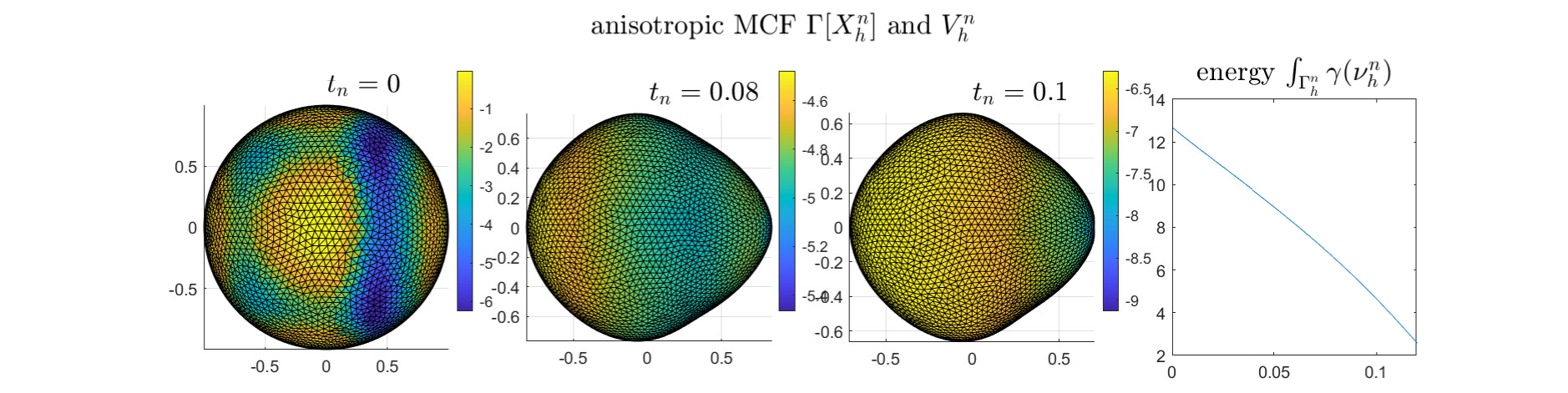}
	
	\caption{Numerical solutions of the anisotropic MCF with the asymmetric function \eqref{eq:asymmetric anisotropy} using the BDF2 / quadratic ESFEM discretization ($\textnormal{dof} = 3882$ and $\tau = 10^{-3}$), and the decay of the anisotropic energy.}
	\label{fig:asymmetric}
\end{figure}

\clearpage

\subsection{Experiments with the stabilized algorithm}
\label{section:regularization experiments}

We have performed some numerical experiments using the stabilization of Section~\ref{section:regularization}:
\begin{equation*}
	\widehat{\ga''}(\nu) := \ga''(\nu) + \nu \nu^T .
\end{equation*}

In Figure~
\ref{fig:regularization_comparison 0.05} we compare the numerical solutions without and with the above stabilization.

\begin{figure}[htp]
	\centering
	
	{\small $\eps = 0.1$}
	
	\includegraphics[width=\textwidth, trim={85 10 45 25}, clip]{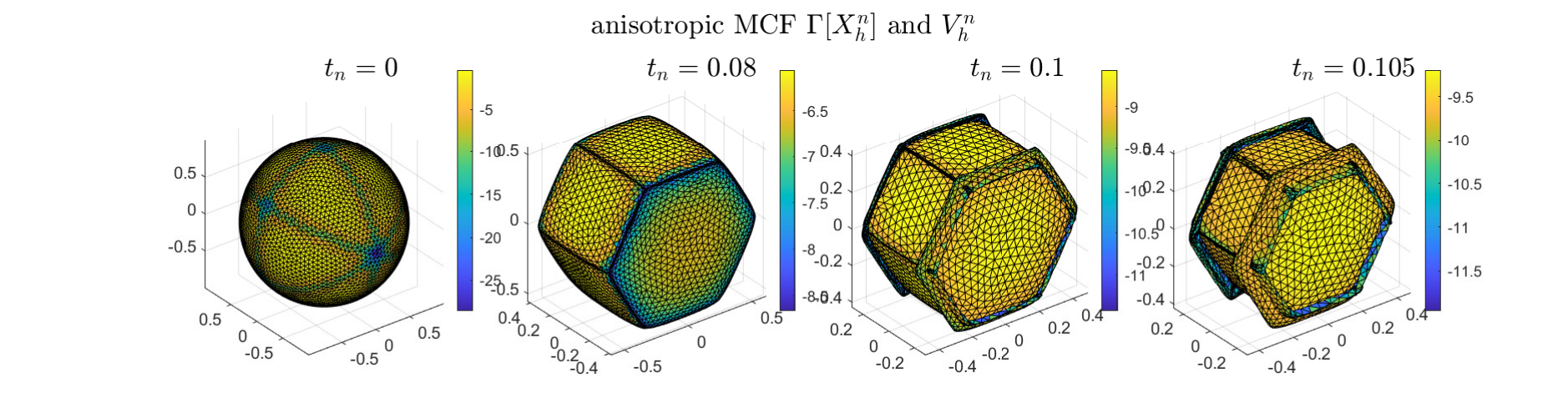}
	
	\includegraphics[width=\textwidth, trim={85 10 45 25}, clip]{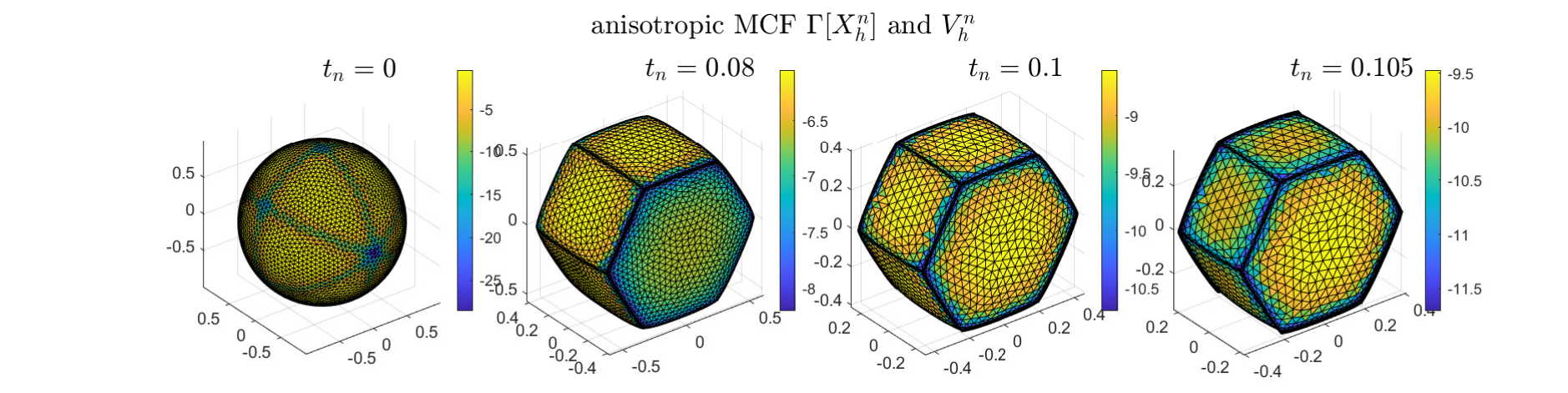}
	
%
%
	\vspace{4.2mm} 

	\centering
	{\small $\eps = 0.05$}
	\includegraphics[width=\textwidth, trim={85 10 45 25}, clip]{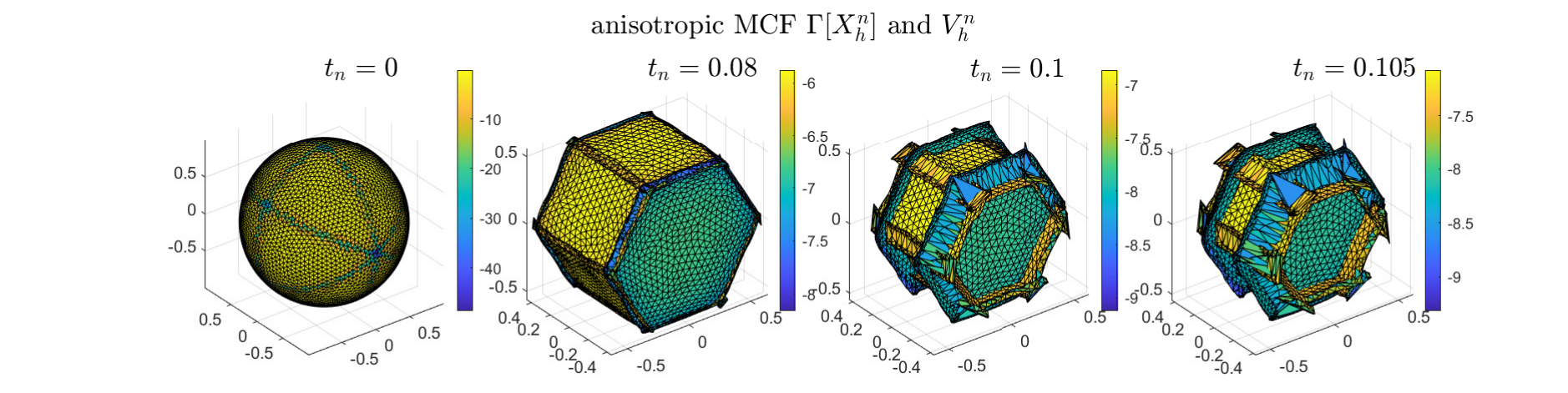}
	
	\includegraphics[width=\textwidth, trim={85 10 45 25}, clip]{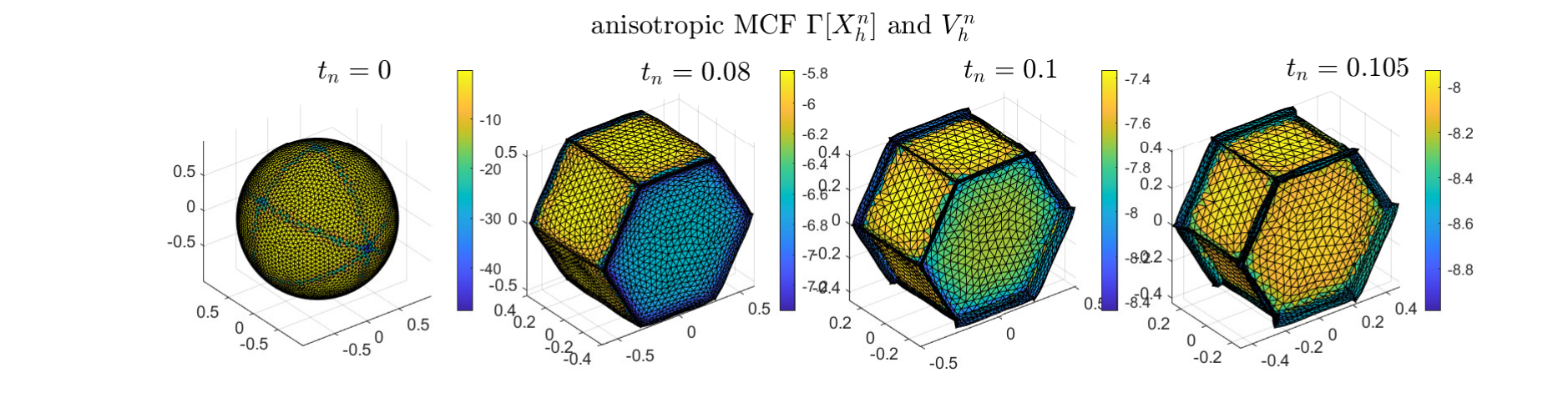}
	
	\caption{Stabilization effect: comparing numerical solutions of with hexagonal anisotropy ($\eps = 0.1$ and $\eps = 0.05$) odd rows without, and even rows with stabilization.}
	\label{fig:regularization_comparison 0.05}
\end{figure}

\clearpage

\section*{Acknowledgments}

The work of Harald Garcke is supported by the DFG Research Training Group 2339 \emph{IntComSin} -- Project-ID 321821685.

The work of Bal\'azs Kov\'acs is funded by the Heisenberg Programme of the Deutsche Forschungsgemeinschaft (DFG, German Research Foundation) -- Project-ID 446431602,
and by the DFG Research Unit FOR 3013 \textit{Vector- and tensor-valued surface PDEs} (BA2268/6–1).

\appendix

\section{Proofs of interchange formulas}

The following  interchange formulas for coordinate-wise differential operators were proved in \cite[Lemma~2.4 and 2.6]{DziukKronerMuller}:
\begin{subequations}
	\begin{align}
		\label{eqA:interchange formula - basic}
		\D_i \D_k u = &\ \D_k \D_i u + A_{kl} \D_l u \nu_i - A_{il} \D_l u \nu_k , \\
		\label{eqA:interchange formula - mat-D}
		\mat (\D_i u) 
		= &\ \D_i (\mat u) - \big( \D_i v_j - \nu_k \nu_i \D_j v_k \big) \D_j u .
	\end{align}
\end{subequations}

The following lemma translates these to surface gradient and surface divergence operators. We remind that we use the convention that the divergence of a matrix is computed column-wise.
\begin{lemma}[interchange formulas]
\label{lemma:interchange formulas}
	Let $u \colon \Ga(t) \to \R$ and $w \colon \Ga(t) \to \R^{d+1}$ be $C^2$ functions on a $C^2$ surface $\Ga(t)$ evolving with velocity $v$. Then the following identities hold
	\begin{align}
		\label{eqA:interchange formula - Hessian-transpose}
		\nbg^2 u = &\ \big( \nbg^2 u \big)^T + \nu \otimes (A \nbg u) - (A \nbg u) \otimes \nu , \\
		\label{eqA:interchange formula - grad-mat}
		\mat \big( \nbg u \big) = &\ \nbg ( \mat u ) - \big( \nbg v  - \nu \nu^T (\nbg v)^T \big) \nbg u , \\
		\label{eqA:interchange formula - grad-div}
		\nbg \big( \nbg \cdot w \big) = &\ \laplace_\Ga w + (A : \nbg w) \nu - A (\nbg w \nu) , \\
		\label{eqA:interchange formula - div-mat}
		\mat \big( \nbg \cdot w \big) = &\ \nbg \cdot ( \mat w ) - (\nbg v)^T : \nbg w + \big(\nbg v \nu \nu^T\big) : \nbg w .
	\end{align}
\end{lemma}
\begin{proof}
(i) The first identity follows directly from \eqref{eqA:interchange formula - basic}.

(ii) The second identity follows from \eqref{eqA:interchange formula - mat-D}.

(iii) For the third identity we use \eqref{eqA:interchange formula - basic} and the fact that $A = \nbg \nu$ is a symmetric matrix, combining these gives
\begin{align*}
	\nbg \big( \nbg \cdot w \big) 
	= &\ D_i D_k w_k \\
	= &\ D_k D_i w_k + A_{kj} D_j w_k \nu_i - A_{ij} D_j w_k \nu_k \\
	= &\ D_k D_i w_k + A_{jk} D_j w_k \nu_i - A_{ji} D_j w_k \nu_k \\
	= &\ \nbg \cdot (\nbg w)^T + (A : \nbg w) \nu - A (\nbg w \nu) .
\end{align*}

(iv) For the fourth identity we again use \eqref{eqA:interchange formula - mat-D} and the definition of surface divergence, which together give
\begin{align*}
	\mat \big( \nbg \cdot w \big) 
	= &\ \mat \big( D_i w_i \big) \\
	= &\ \D_i (\mat w_i) - \big( \D_i v_j - \nu_k \nu_i \D_j v_k \big) \D_j w_i \\
	= &\ \nbg \cdot ( \mat w ) - \big( (\nbg v)^T  - \nbg v \nu \nu^T  \big) : \nbg w .
\end{align*}
\qed
\end{proof}


\bibliographystyle{alpha}
\bibliography{evolving_surface_literature}

\end{document}